\pdfoutput=1
\RequirePackage{ifpdf}
\ifpdf 
\documentclass[pdftex]{sigma}
\else
\documentclass{sigma}
\fi

\numberwithin{equation}{section}

\newtheorem{Theorem}{Theorem}[section]
\newtheorem{Corollary}[Theorem]{Corollary}
\newtheorem{Lemma}[Theorem]{Lemma}
\newtheorem{Proposition}[Theorem]{Proposition}
 { \theoremstyle{definition}
\newtheorem{Definition}[Theorem]{Definition}
\newtheorem{Example}[Theorem]{Example}
\newtheorem{Remark}[Theorem]{Remark} }

\begin{document}

\newcommand{\arXivNumber}{2011.05886}

\renewcommand{\PaperNumber}{054}

\FirstPageHeading

\ShortArticleName{Nonsymmetric Macdonald Superpolynomials}

\ArticleName{Nonsymmetric Macdonald Superpolynomials}

\Author{Charles F.~DUNKL}

\AuthorNameForHeading{C.F.~Dunkl}

\Address{Department of Mathematics, University of Virginia, \\ PO Box 400137, Charlottesville VA 22904-4137, USA}
\Email{\href{mailto:cfd5z@virginia.edu}{cfd5z@virginia.edu}}
\URLaddress{\url{https://uva.theopenscholar.com/charles-dunkl}}

\ArticleDates{Received November 16, 2020, in final form May 13, 2021; Published online May 23, 2021}

\Abstract{There are representations of the type-A Hecke algebra on spaces of polynomials in anti-commuting variables. Luque and the author [\textit{S\'em. Lothar. Combin.} \textbf{66} (2012), Art.~B66b, 68~pages, arXiv:1106.0875] constructed nonsymmetric Macdonald polynomials taking values in arbitrary modules of the Hecke algebra. In~this paper the two ideas are combined to define and study nonsymmetric Macdonald polynomials taking values in the aforementioned anti-commuting polynomials, in other words, superpolynomials. The modules, their orthogonal bases and their properties are first derived. In~terms of the standard Young tableau approach to representations these modules correspond to hook tableaux. The details of the Dunkl--Luque theory and the particular application are presented. There is an inner product on the polynomials for which the Macdonald polynomials are mutually orthogonal. The squared norms for this product are determined. By using techniques of Baker and Forrester [\textit{Ann. Comb.} \textbf{3} (1999), 159--170, arXiv:q-alg/9707001] symmetric Macdonald polynomials are built up from the nonsymmetric theory. Here ``symmetric'' means in the Hecke algebra sense, not in the classical group sense. There is a concise formula for the squared norm of the minimal symmetric polynomial, and some formulas for anti-symmetric polynomials. For both symmetric and anti-symmetric polynomials there is a~factorization when the polynomials are evaluated at special points.}


\Classification {33D56; 20C08; 05E05}

\vspace{-1mm}

\section{Introduction}

Nonsymmetric Macdonald~\cite{Macdonald2001} polynomials are simultaneous
eigenfunctions of a set of mutually commuting operators derived from an action
of the type-$A$ Hecke algebra on the space of~polynomials in $N$ variables.
They are significantly different from the symmetric Mac\-do\-nald polynomials in
the technique of their respective definitions and yet Baker and Forrester~\cite{BakerForrester97} established a strong relation between them. In~the analogous
theory of nonsymmetric Jack polynomials Griffeth~\cite{Griffeth2010} constructed such
polynomials which take values in modules of the underlying groups,
specifically the complex reflection groups in the infinite family $G(\ell,p,N)$. These polynomials constitute a standard module of the
rational Cherednik algebra. Luque and the author~\cite{DunklLuque2012} extended the
theory of nonsymmetric Macdonald polynomials in the direction suggested by
Griffeth's work by studying polynomials taking values in modules of the Hecke
algebra. The development relies on exploiting standard Young tableaux and the
Yang--Baxter graph technique of Lascoux~\cite{Lascoux2001}.

The superpolynomials considered here are generated by $N$ anti-commuting and
$N$ commuting variables. By defining representations of the Hecke algebra on
anti-commuting variables the theory of vector-valued nonsymmetric Macdonald
polynomials is applied to define and ana\-lyze superpolynomials. There is a
theory of symmetric Macdonald superpolynomials initiated by~Blondeau-Fournier,
Desrosiers, Lapointe, and Mathieu~\cite{BlondeauFournieretal2012} with further developments on~norm and special point values by Gonz\'{a}lez and Lapointe~\cite{GonzalezLapointe2020}.
Their approach and definitions are based on differential operators and linear
combinations of the classical nonsymmetric Macdonald polynomials, whose
coefficients involve anti-commuting variables. The theory developed in the
present paper is different due to the method of using anti-commuting variables
to form Hecke algebra modules.

Nonsymmetric Macdonald polynomials associated with general root systems were
intensively studied by Cherednik~\cite{Cherednik95}. By specializing to root systems
of type $A$ it becomes possible to develop more detailed relations, formulas
and structure. In~particular, the papers of Noumi and Mimachi~\cite{MimmachiNoumi98},
Baker and Forrester~\cite{BakerForrester97} provide important background for the present
paper. Note that some authors use different axioms for the quadratic relations
of the Hecke algebra, such as $\big(T-t^{1/2}\big)\big(T+t^{-1/2}\big) =0$, rather than $(T-t)(T+1)=0$.

The theory of Hecke algebras of type $A$ and their representations is briefly
described in~Sec\-tion~\ref{HA} and then applied to modules of polynomials in
anti-commuting variables. In~general the irreducible representations are
constructed as spans of standard Young tableaux whose shape corresponds to a
fixed partition of~$N$. In~the present situation it is the hook tableaux which
arise. The basis vectors are constructed and the important transformation
formulas are stated. There is an inner product in which the generators of the
Hecke algebra are self-adjoint which leads to evaluation of the squared norms
of the basis elements.

In~Section~\ref{NSMP} the theory of vector-valued nonsymmetric Macdonald
polynomials developed in~\cite{DunklLuque2012} is applied to produce superpolynomials,
considered as polynomials taking values in modules of anti-commuting
variables. The main results are stated without proofs but some important
details are carefully worked out. In~\cite{Dunkl2019} the author constructed an
inner product in which the nonsymmetric Macdonald polynomials are mutually
orthogonal, in the general vector-valued situation. This structure is worked
out for the superpolynomials in Section~\ref{SymFM} and the squared norms
are computed. In~Section~\ref{SyMP} the techniques of Baker and Forrester~\cite{BakerForrester97} are used to produce supersymmetric Macdonald polynomials, and
the squared norms. From results of~\cite{DunklLuque2012} the labels of these
polynomials correspond to the superpartitions of Desrosiers, Lapointe, and
Mathieu~\cite{Derosiersetal2003}. It has to be emphasized that in this paper the meaning
of symmetric is with respect to the Hecke algebra, not the symmetric group.
Also the squared norm of the lowest degree supersymmetric polynomial is
determined~-- the formula is more elegant than the general formula; its
calculation is able to use telescoping arguments for simplifications. There is
a derivation of formulas for antisymmetric Macdonald polynomials in Section~\ref{antiSy}. In~the conclusion some further topics of~inves\-tigation, such as
evaluation at special points, are discussed.

\section[The Hecke algebra of type A]{The Hecke algebra of type $\boldsymbol A$}\label{HA}

\subsection{Definitions and Jucys--Murphy elements}

The Hecke algebra $\mathcal{H}_{N}(t) $ of type $A_{N-1}$ with
parameter $t$ is the associative algebra over an~exten\-sion field of
$\mathbb{Q}$, generated by $\big\{ T_{1},\dots,T_{N-1}\big\} $ subject
to the braid relations%
\begin{subequations}
\begin{gather}
T_{i}T_{i+1}T_{i}=T_{i+1}T_{i}T_{i+1},\qquad 1\leq i<N-1,\label{Hbrd}
\\
T_{i}T_{j}=T_{j}T_{i},\qquad \vert i-j\vert \geq2, \label{Hbrd2}%
\end{gather}
and the quadratic relations%
\end{subequations}
\begin{gather}
(T_{i}-t)(T_{i}+1) =0,\qquad 1\leq i<N, \label{Hquad}%
\end{gather}
where $t$ is a generic parameter (this means $t^{n}\neq1$ for $2\leq n\leq
N$). The quadratic relation implies $T_{i}^{-1}=\frac{1}{t}(T_{i}+1-t)$. There is a commutative set in $\mathcal{H}_{N}(t)$ of \textit{Jucys--Murphy elements} defined by $\omega_{N}=1$, $\omega_{i}=t^{-1}T_{i}\omega_{i+1}T_{i}$ for $1\leq i<N$, that is,%
\begin{gather*}
\omega_{i}=t^{i-N}T_{i}T_{i+1}\cdots T_{N-1}T_{N-1}T_{N-2}\cdots T_{i}.%
\end{gather*}
Simultaneous eigenvectors of $\{\omega_{i}\} $ form bases of
irreducible representations of the algebra. The symmetric group $\mathcal{S}_{N}$ is the group of permutations of $\{1,2,\dots,N\} $ and is
generated by the simple reflections (adjacent transpositions) $\{
s_{i}\colon 1\leq i<N\}$, where $s_{i}$ interchanges $i$, $i+1$ and fixes the
other points (the $s_{i}$ satisfy the braid relations and $s_{i}^{2}=1$).
There is a linear isomorphism $\mathbb{Z}\mathcal{S}_{N}\rightarrow
\mathcal{H}_{N}(t) $ given by $\sum_{u\in\mathcal{S}_{N}%
}a_{u}u\rightarrow\sum_{u\in\mathcal{S}_{N}}a_{u}T(u) $,
where $T(u) =T_{i_{1}}\cdots T_{i_{\ell}}$ with $u=s_{i_{1}%
}\cdots s_{i_{\ell}}$ being a shortest expression for $u$ (in fact
$\ell=\#\{(i,j) \colon i<j,\,u(i) >u(j)\}$); $T(u)$ is well-defined because of the
braid relations (see~\cite{DipperJames86}).

\subsection{Modules of anti-commuting variables}

Consider polynomials in $N$ anti-commuting (fermionic) variables $\theta
_{1},\theta_{2},\dots,\theta_{N}$. They satisfy $\theta_{i}^{2}=0$ and
$\theta_{i}\theta_{j}+\theta_{j}\theta_{i}=0$ for $i\neq j$. The basis for
these polynomials consists of monomials labeled by subsets of $\{1,2,\dots,N\}$:%
\begin{gather*}
\phi_{E}:=\theta_{i_{1}}\cdots\theta_{i_{m}},\qquad
E=\{i_{1},i_{2},\dots,i_{m}\},\qquad
1\leq i_{1}<i_{2}<\cdots<i_{m}\leq N.
\end{gather*}

The polynomials have coefficients in an extension field of $\mathbb{Q}(q,t)$ with transcendental $q$, $t$, or~gene\-ric $q$, $t$ satisfying
$q,t\neq0$, $q^{a}\neq1$, $q^{a}t^{n}\neq1$ for $a\in\mathbb{Z}$ and
$n\neq2,3,\dots,N$.

\begin{Definition}
$\mathcal{P}:=\operatorname{span}\big\{\phi_{E}\colon E\subset\{1,\dots,N\}\big\}$ and $\mathcal{P}_{m}:=\operatorname{span}\big\{\phi_{E}\colon \#E=m\big\}$ for $0\leq m$ $\leq N$. The fermionic degree of~$\phi_{E}$ is $\#E$.
\end{Definition}

Some utility formulas are used for working with $\{\phi_{E}\}$.

\begin{Definition}
For a subset $E\subset\{1,2,\dots,N\}$ and $1\leq i<N$ let
\begin{gather*}
E^{C} :=\let\{1,2,\dots,N\} \backslash E,
\\
\operatorname{inv}(E) :=\#\big\{(i,j) \in E\times E^{C}\colon i<j\big\},
\\
s_{i}\phi_{E} =\phi_{s_{i}E}=\phi_{( E\backslash\{i\}) \cup\{i+1\}},\qquad
(i,i+1) \in E\times E^{C},
\\
s_{i}\phi_{E} =\phi_{s_{i}E}=\phi_{(E\backslash\{i+1\}) \cup\{i\}},\qquad
(i,i+1) \in E^{C}\times E.
\end{gather*}
\end{Definition}
\noindent
(When $\{i,i+1\} \subset E$ or $\subset E^{C}$ then $s_{i}E=E$
and $s_{i}\phi_{E}=\phi_{E}$.) Introduce a representation of~$\mathcal{H}_{N}(t)$ on~$\mathcal{P}$.

\begin{Definition}\label{defTi}
For $1\leq i<N$%
\begin{gather*}
T_{i}\phi_{E}=
\begin{cases}
-\phi_{E}, &\{i,i+1\} \subset E,
\\
t\phi_{E}, &\{i,i+1\} \subset E^{C},
\\
s_{i}\phi_{E}, &(i,i+1) \in E\times E^{C},
\\
(t-1) \phi_{E}+ts_{i}\phi_{E}, &(i,i+1) \in E^{C}\times E.
\end{cases}
\end{gather*}

\end{Definition}

\begin{Proposition}
The operators $\{T_{i}\}$ satisfy the braid and quadratic relations~\eqref{Hbrd} and~\eqref{Hquad}.
\end{Proposition}

\begin{proof}
It suffices to verify that $T_{1}T_{2}T_{1}=T_{2}T_{1}T_{2}$ and $(T_{1}-t) (T_{1}+1) =0$ on the spaces $\operatorname{span}\{\theta_{1},\theta_{2},\theta_{3}\} $
and~$\operatorname{span}\{\theta_{2}\theta_{3},\theta_{1}\theta_{3},\theta_{1}\theta_{2}\}$. The~relations are trivially satisfied on $\operatorname{span}\{1\}$ and $\operatorname{span}\{\theta_{1}\theta_{2}\theta_{3}\}$.
\end{proof}

\begin{Remark}
For symbolic computation and to verify the previous proposition use%
\begin{gather*}
T_{i}f(\theta_{1},\dots,\theta_{N}) =tf+(t\theta_{i}-\theta_{i+1}) \bigg(\frac{\partial}{\partial\theta_{i+1}}-\frac{\partial}{\partial\theta_{i}}\bigg) f
 -\big(t\theta_{i}^{2}+\theta_{i+1}^{2}\big) \frac{\partial^{2}}{\partial\theta_{i}\partial\theta_{i+1}}f,
\\
T_{i}^{-1}f(\theta_{1},\dots,\theta_{N}) =\frac{1}{t}f+\bigg(\theta_{i}-\frac{1}{t}\theta_{i+1}\bigg) \bigg(\frac{\partial}{\partial\theta_{i+1}}-\frac{\partial}{\partial\theta_{i}}\bigg)f
 -\bigg(\theta_{i}^{2}+\frac{1}{t}\theta_{i+1}^{2}\bigg) \frac{\partial^{2}}{\partial\theta_{i}\partial\theta_{i+1}}f,
\end{gather*}
with the partial derivatives being formal (the order of variables is ignored).
\end{Remark}

There is a symmetric bilinear form on $\mathcal{P}$ which is positive-definite
for $t>0$ and in which ~$T_{i}$ is self-adjoint for $1\leq i<N$. The purpose
of the form is to make the simultaneous eigenvectors of $\{\omega_{i}\} $ mutually perpendicular.

\begin{Definition}\label{phiform}
For $E,F\subset\{1,2,\dots,N\}$ define
$\left\langle \phi_{E},\phi_{F}\right\rangle =\delta_{E,F}t^{-\operatorname{inv}(E)}$ and extend the form to $\mathcal{P}$ by linearity.
\end{Definition}

\begin{Proposition}
Suppose $f,g\in\mathcal{P}$ and $1\leq i<N$
then $\left\langle T_{i}f,g\right\rangle =\left\langle f,T_{i}g\right\rangle $.
\end{Proposition}

\begin{proof}
It suffices to consider $T_{1}$. Let $F\subset\left\{ 3,4,\dots,N\right\}
$. Then $T_{1}\phi_{F}=t\phi_{F}$ and $T_{1}\phi_{\left\{ 1,2\right\} \cup
F}=-\phi_{\left\{ 1,2\right\} \cup F}$. Let $F_{i}=F\cup\{i\}
$ for $i=1,2$ and $T_{1}\phi_{F_{1}}=\phi_{F_{2}}$ and $T_{1}\phi_{F_{2}%
}=(t-1) \phi_{F_{2}}+t\phi_{F_{1}}$ so that%
\begin{gather*}
\left\langle T_{1}\phi_{F_{1}},\phi_{F_{2}}\right\rangle =\left\langle
\phi_{F_{2}},\phi_{F_{2}}\right\rangle =t^{-\operatorname{inv}\left( F_{2}\right)
},\\
\left\langle \phi_{F_{1}},T_{1}\phi_{F_{2}}\right\rangle =\left\langle
\phi_{F_{1}},(t-1) \phi_{F_{2}}+t\phi_{F_{1}}\right\rangle
=t\left\langle \phi_{F_{1}},\phi_{F_{1}}\right\rangle =t^{1-\operatorname{inv}%
\left( F_{1}\right) },
\end{gather*}
and $\operatorname{inv}(F_{1}) =\operatorname{inv}(F_{2}) +1$
by counting the pair $(1,2) \in F_{1}\times F_{1}^{C}$.
\end{proof}

\begin{Corollary}\label{omadjt}
If $f,g\in\mathcal{P}$ and $1\leq i\leq N$ then $\left\langle
\omega_{i}f,g\right\rangle =\left\langle f,\omega_{i}g\right\rangle$.
\end{Corollary}

\begin{proof}
This follows from $\omega_{N}=1$ and $\omega_{i}=t^{-1}T_{i}\omega_{i+1}T_{i}$ for $i<N$.
\end{proof}

There are two degree-changing linear maps which commute with the Hecke algebra action.

\begin{Definition}\label{defMD}
For $n\in\mathbb{Z}$ set $\sigma(n) :=(-1)^{n}$ and for $E\subset\{ 1,2,\dots,N\}$,
$1\leq i\leq N$ set $s(i,E) :=\#\{ j\in E\colon j<i\}$. Define
the operators $\partial_{i}$ and $\widehat{\theta}_{i}$ by $\partial_{i}
\theta_{i}\phi_{E}=\phi_{E}$, $\partial_{i}\phi_{E}=0$ and~$\widehat{\theta}_{i}\phi_{E}=\theta_{i}\phi_{E}=\sigma(s(i,E))\phi_{E\cup\{i\}}$
for~$i\notin E$, while~$\widehat{\theta}_{i}\phi_{E}=0$ for~$i\in E$ $\big($also $i\in E$ implies $\phi_{E}=\sigma(s(i,E)) \theta_{i}\phi_{E\backslash\{i\}}$ and~$\partial_{i}\phi_{E}=\sigma(s(i,E))
\phi_{E\backslash\{i\}}\big)$. Define $M:=\sum_{i=1}^{N}\widehat{\theta}_{i}$ and~$D:=\sum_{i=1}^{N}t^{i-1}\partial_{i}$.
\end{Definition}

By direct computation one can show that $\widehat{\theta}_{i}\partial_{j}=-\partial_{j}\widehat{\theta}_{i}$ for $i\neq j$.

\begin{Proposition}\label{commMD}
$M$ and $D$ commute with $T_{i}$ for $1\leq i<N.$
\end{Proposition}

\begin{proof}It follows from the definitions that $\partial_{i}$ and $\widehat{\theta}_{j}$
commute with $T_{i}$ when $j<i$ or $j>i+1$. It suffices to show $\partial
_{1}+t\partial_{2}$ and $\widehat{\theta}_{1}+\widehat{\theta}_{2}$ commute
with $T_{1}$ applied to $p_{1}:=\phi_{F}$, $p_{2}:=( \theta_{1}
+\theta_{2}) \phi_{F}$, $~p_{3}:=( t\theta_{1}-\theta_{2})
\phi_{F}$, $p_{4}:=\theta_{1}\theta_{2}\phi_{F}$ with $1,2\notin F$. Then
$T_{1}p_{i}=tp_{i}$ for $i=1,2$, $T_{1}p_{i}=-p_{i}$ for $i=3,4$ and
\begin{gather*}
( \partial_{1}+t\partial_{2}) [ p_{1},p_{2},p_{3},p_{4}]  =[ 0,( t+1) p_{1},0,-p_{3}] ,\\
( \widehat{\theta}_{1}+\widehat{\theta}_{2}) [ p_{1},p_{2},p_{3},p_{4}] =[ p_{2},0,-( t+1)p_{4},0] .
\end{gather*}
This concludes the proof.
\end{proof}

It is clear that $D^{2}=0=M^{2}$. For $n=0,1,2,\dots$ let $[n]_{t}:=\frac{1-t^{n}}{1-t}$ and $[n]_{t}!:=[1]_{t}[2]_{t}\cdots[n]_{t}$.

\begin{Proposition}\label{MDDM}
$MD+DM=[N]_{t}$.
\end{Proposition}

\begin{proof}
Fix $\phi_{E}$, $\#E=m$; $\partial_{j}\phi_{E}=\sigma(s(j,E)) \phi_{E\backslash\{j\}}$, then
$\widehat{\theta}_{j}\partial_{j}\phi_{E}=\sigma(s(j,E)) \theta_{j}\phi_{E\backslash\{j\}}=\phi_{E}$ thus the
coefficient of $\phi_{E}$ in $MD$ is $\sum_{j\in E}t^{j-1}$. Also $M\phi_{E}=\sum_{i\notin E}\theta_{i}\phi_{E}$ and the coefficient of~$\phi_{E}$ in $DM\phi_{E}$ is $\sum_{i\notin E}t^{i-1}$ so that the
coefficient of $\phi_{E}$ in $MD\!+\!DM$ is $\sum_{j=1}^{N}\!t^{j-1}=[N]_{t}$. Suppose $i\in E$, $j\notin E$ then $\widehat{\theta}_{j}\phi_{E\backslash\{i\}}$ appears in $MD$ with coefficient
$t^{i-1}\sigma(s(i,E)) $ while $t^{i-1}\partial_{i}\theta_{j}\phi_{E}=\sigma(s(i,E))
t^{i-1}\partial_{i}\theta_{j}\theta_{i}\phi_{E\backslash\{i\}
}=-\sigma(s(i,E)) t^{i-1}\widehat{\theta}_{j}\phi_{E\backslash\{i\} }$, and this term is canceled out in~\mbox{$MD+DM$}.
\end{proof}

\subsection[Representations of H N (t)]{Representations of $\boldsymbol{\mathcal{H}_{N}(t)}$}

These representations correspond to partitions of $N$, namely $\lambda=\left(
\lambda_{1},\dots,\lambda_{N}\right) \in\mathbb{N}_{0}^{N}$ with
$\lambda_{1}\geq\lambda_{2}\geq\cdots\geq\lambda_{N}$ and $\sum_{i=1}%
^{N}\lambda_{i}=N$. The length of $\lambda$ is $\ell(\lambda)
=\max\{i\colon \lambda_{i}\geq1\}$. There is a graphical device to
picture $\lambda$, called the \textit{Ferrers diagram}, which has boxes at
$\{(i,j) \colon 1\leq i\leq\ell(\lambda)\,,1\leq j\leq\lambda_{i}\}$ (integer points). A reverse standard tableau
(RSYT) is a filling of the Ferrers diagram with the numbers $\{1,2,\dots,N\}$ such that the entries decrease in each row and in each
column. The relevant representation of $\mathcal{H}_{N}(t) $ is
defined on the span of the RSYT's of shape $\lambda$ in such a way that
$\omega_{i}Y=t^{c(i,Y) }Y$ for $1\leq i\leq N$, where $Y[a,b] =i$, $c(i,Y) =b-a$ ($b-a$ is called the
\textit{content} of $[a,b]$), and $Y$ is a RSYT of shape
$\lambda$. In~the present work only hook tableaux will occur, namely
partitions of the form $\lambda=(N-n,1^{n})$ (the part $1$ is
repeated $n$ times), so that $\ell(\lambda) =n+1$.

We will show that $\mathcal{P}_{m}$ is a direct sum of the $\mathcal{H}_{N}(t)$-modules corresponding to $(N-m,1^{m})$ and $(N+1-m,1^{m-1})$. Here is a structure for labeling the
$\phi_{E}$ of interest.

\begin{Definition}
Let $\mathcal{Y}_{0}:=\{E\colon \#E=m+1,\,N\in E\} $ and
$\mathcal{Y}_{1}:=\{E\colon \#E=m-1,\,N\notin E\}$.
\end{Definition}

These sets are associated to RSYT's of shape $\big(N-m,1^{m}\big)$
and $\big(N-m+1,1^{m-1}\big)$ respectively, and this
correspondence will be used to define content vectors for~$E$.

\begin{Definition}
Suppose $E\in\mathcal{Y}_{0}$ and $E=\{i_{1},\dots,i_{m},i_{m+1}\}$,
 $E^{C}=\{j_{1},\dots,j_{N-m-1}\}$ with $i_{1}<i_{2}<\cdots<i_{m+1}=N$ and $j_{1}<j_{2}<\cdots$ then
$Y_{E}$ is the RSYT of shape $(N-m,1^{m})$ given by
$Y_{E}[k,1] =i_{m+2-k}$ for $1\leq k\leq m+1$, and $Y_{E}[1,k] =j_{N-m+1-k}$ for $2\leq k\leq N-m$. Suppose $E\in\mathcal{Y}_{1}$
and $E=\{i_{1},\dots,i_{m-1}\}$, $E^{C}=\{j_{1},\dots,j_{N-m+1}\}$ with $i_{1}<i_{2}<\cdots$ and $j_{1}<j_{2}<\cdots<j_{N-m+1}=N$ then $Y_{E}$ is the RSYT of shape $(N-m+1,1^{m-1})$ given by~$Y_{E}[k,1] =i_{m+1-k}$ for~$2\leq k\leq m$,
$Y_{E}[1,k] =j_{N-m+2-k}$ for~$1\leq k\leq N-m+1$. In~both cases
define the content vector $c(i,E) =c\big( i,Y_{E}\big)$
for $1\leq i\leq N$.
\end{Definition}

For space-saving convenience the RSYT's are displayed in two rows, with the
second row consisting of the entries $Y_{E}[2,1],\,Y_{E}[3,1],\dots$\,. Recall the \textit{content} of cell $[i,j]$ is $j-i$.

As example let $N=8$, $m=3$, $E=\{2,5,7,8\}$ then
\begin{gather*}
Y_{E}=%
\begin{bmatrix}
8 & 6 & 4 & 3 & 1
\\
\cdot & 7 & 5 & 2 &
\end{bmatrix}
\end{gather*}
and $[c(i,E)]_{i=1}^{8}=[4,-3,3,2,-2,1,-1,0]$.

We will construct for each $E\in\mathcal{Y}_{0}\cup\mathcal{Y}_{1}$ a
polynomial $\tau_{E}\in\mathcal{P}_{m}$ such that $\omega_{i}\tau
_{E}=t^{c(i,E) }\tau_{E}$ for~$1\leq i\leq N$. To start let
$E_{0}:=\{N-m,N-m+1,\dots,N\} \in\mathcal{Y}_{0}$. Then
\begin{gather*}
Y_{E_{0}} =\begin{bmatrix}
N & N-m-1 & N-m-2 & \cdots & \cdots & 1
\\
\cdot & N-1 & N-2 & \cdots & N-m &
\end{bmatrix}\!,
\\
\left[c(i,E_{0})\right]_{i=1}^{N} =[N-m-1,N-m-2,\dots,1,-m,1-m,\dots,-1,0].
\end{gather*}

\begin{Theorem}\label{psi0eigen}
Let $\psi_{0}=D\phi_{E_{0}}\in\ker D\cap\mathcal{P}_{m}$,
then $\omega_{i}\psi_{0}=t^{c(i,E_{0})}\psi_{0}$ for $1\leq i\leq N$.
\end{Theorem}

\begin{proof}
If $N-m\leq i<N$ then $T_{i}\phi_{E_{0}}=-\phi_{E_{0}}$ and so $T_{i}\psi
_{0}=-\psi_{0}$, because $T_{i}D=DT_{i}$ and $\psi_{0}=D\phi_{E_{0}}$, thus
$\omega_{i}\psi_{0}=t^{i-N}\psi_{0}$. It is clear that $T_{i}\psi_{0}%
=t\psi_{0}$ for $1\leq i<N-m-1$ and so it remains to prove $\omega_{N-m-1}%
\psi_{0}=t\psi_{0}$. (The remaining part of the argument is straightforward,
and is at the end of this proof; for example $\omega_{N-m-2}\psi_{0}%
=t^{-1}T_{N-m-2}\omega_{N-m-1}T_{N-m-2}\psi_{0}=T_{N-m-2}\omega_{N-m-1}%
\psi_{0}=t^{2}\psi_{0}$.) Let $F:=\{N-m-1,N-m,\dots,N\} $ and
$F_{j}:=F\backslash\{j\}$, $p_{j}=\phi_{F_{j}}$, (so that
$p_{N-m-1}=\phi_{E_{0}}$) then $T_{i}p_{j}=-p_{j}$ if $i>j$ or $N-m-1\leq
i<j-1$, $T_{j}p_{j}=(1-t) p_{j}+tp_{j+1}$ and $T_{j}%
p_{j+1}=p_{j}$. To set up an induction argument let $U_{N-m-1}=T_{N-m-1}$ and
$U_{i+1}=T_{i+1}U_{i}$ for $i<N-1$. We claim%
\[
U_{i}\phi_{E_{0}}=t^{i-N+m+2}p_{i+1}+(t-1) \sum_{j=N-m-1}^{i}(-1)^{i-j}t^{j-N+m+1}p_{j}.
\]
At the start of the induction $T_{N-m-1}p_{N-m-1}=tp_{N-m}+(t-1)
p_{N-m-1}$. Suppose the formula holds for $i$ then%
\begin{gather*}
U_{i+1}\phi_{E_{0}}=t^{i-N+m+2}T_{i+1}p_{i+1}+(t-1)\sum_{j=N-m-1}^{i}(-1)^{i-j}t^{j-N+m+1}T_{i+1}p_{j}
\\ \hphantom{U_{i+1}\phi_{E_{0}}}
{}=t^{i-N+m+2}(tp_{i+2}+(t-1) p_{i+1}) -(t-1) \sum_{j=N-m-1}^{i}(-1)^{i-j}t^{j-N+m+1}p_{j}
\\ \hphantom{U_{i+1}\phi_{E_{0}}}
{}=t^{i-N+m+3}p_{i+2}+(t-1) \sum_{j=N-m}^{i+1}(-1)^{j-i+1}t^{j-N+m+1}p_{j},
\end{gather*}
as is to be shown. We also need
\[
D\phi_{F}=\sum_{j=N-m-1}^{N-1}(-1)^{j-N+m+1}t^{j-1}p_{j}+(-1)^{m+1}t^{N-1}p_{N}.
\]
Thus%
\begin{gather*}
U_{N-1}\phi_{E_{0}} =t^{m+1}p_{N}+(t-1) \sum_{j=N-m-1}^{N-1}(-1)^{N-1-j}t^{j-N+m+1}p_{j}
\\ \hphantom{U_{N-1}\phi_{E_{0}}}
{} =t^{m+1}p_{N}+(t-1) (-1)^{m}t^{-N+m+2}\big\{D\phi_{F}-(-1)^{m+1}t^{N-1}p_{N}\big\}
\\ \hphantom{U_{N-1}\phi_{E_{0}}}
{} =t^{m+2}p_{N}+(t-1) (-1)^{m}t^{-N+m+2}D\phi_{F}.
\end{gather*}
Then $U_{N-1}\psi_{0}=DU_{N-1}\phi_{E_{0}}=t^{m+2}Dp_{N}$ and $T_{N-m-1}
T_{N-m}\cdots T_{N-1}U_{N-1}\psi_{0}=t^{m+2}Dp_{N-m-1}$ $=t^{m+2}\psi_{0}$ since
$T_{j}p_{j+1}=p_{j}$ for $N-m-1\leq j<N$. Hence $\omega_{N-m-1}\psi
_{0}=t^{N-m-1-N}t^{m+2}\psi_{0}$ $=t\psi_{0}$. It follows that%
\begin{align*}
\omega_{i}\psi_{0} & =t^{i-N+m+1}T_{i}\cdots T_{N-m-2}\omega_{N-m-1}%
T_{N-m-2}\cdots T_{i}\psi_{0}
\\
& =t^{i-N+m+1}t^{1+2(N-m-1-i)}\psi_{0}=t^{N-m-i}\psi_{0},
\end{align*}
for $1\leq i\leq N-m-1$.
\end{proof}

Turning to the isotype $\left( N-m+1,1^{m-1}\right) $, let $E_{1}:=\{1,2,\dots,m-1\} \in\mathcal{Y}_{1}$ so that%
\begin{gather*}
Y_{E_{1}} =\begin{bmatrix}
N & N-1 & N-2 & \cdots & \cdots & m
\\
\cdot & m-1 & m-2 & \cdots & 1 &
\end{bmatrix}\!,
\\
[c(i,E_{1})]_{i=1}^{N} =[1-m,2-m,\dots,-1,N-m,N-m-1,\dots,1,0] .
\end{gather*}

\begin{Theorem}
\label{eta0eigen}
Let $\eta_{E_{1}}=M\phi_{E_{1}}\in\ker M\cap\mathcal{P}_{m}$.
Then $\omega_{i}\eta_{E_{1}}=t^{c(i,E_{1})}\eta_{E_{1}}$ for~$1\leq i\leq N$.
\end{Theorem}

\begin{proof}
Since $T_{i}\phi_{E_{1}}=t\phi_{E_{1}}$ for $m\leq i<N$ it follows that
$\omega_{i}M\phi_{E_{1}}=t^{N-i}M\phi_{E_{1}}$. Also $T_{i}\phi_{E_{1}}=-\phi_{E_{1}}$ and $T_{i}M\phi_{E_{1}}=-M\phi_{E_{1}}$ for $1\leq i<m-1$ and
so it remains to show $\omega_{m-1}\eta_{E_{1}}=t^{-1}\eta_{E_{1}}$. Let
$F=\{1,2,\dots,m-2\} $ and for $m-1\leq j\leq N$ let
$F_{j}=F\cup\{j\} $ and $p_{j}=\phi_{F_{j}}$ ($\phi_{E_{1}}=p_{m-1}$). Then $T_{j}p_{j}=p_{j+1}$ so that $T_{N-1}T_{N-2}\cdots
T_{m-1}p_{m-1}=p_{N}$. Also $T_{j}p_{j+1}=tp_{j}+(t-1) p_{j+1}$
and~$T_{j}p_{i}=tp_{i}$ for $i>j+1$. By induction we prove that
\[
T_{i}T_{i+1}\cdots T_{N-1}p_{N}=t^{N-i}p_{i}+(t-1) t^{N-i-1}\sum_{j=i+1}^{N}p_{j}.
\]
The formula is valid for $i=N-1$ and assuming it is true for $i$ apply
$T_{i-1}$ to both sides, then the first term becomes $t^{N-i}\left(
p_{i-1}+(t-1) p_{i}\right) $ and the second term is multiplied
by $t$. Substitute $M\phi_{F}=\sum_{j=m-1}^{N}p_{j}$ in the formula with
$i=m-1$ to obtain
\begin{gather*}
T_{m-1}\cdots T_{N-1}p_{N} =t^{N-m+1}p_{m-1}+(t-1)
t^{N-m}\big\{ M\phi_{F}-p_{m-1}\big\}
\\ \hphantom{T_{m-1}\cdots T_{N-1}p_{N} }
{} =t^{N-m}p_{m-1}+(t-1) t^{N-m}M\phi_{F}.
\end{gather*}
Thus $\omega_{m-1}\eta_{E_{1}}=t^{m-1-N}t^{N-m}Mp_{m-1}=t^{-1}\eta_{E_{1}}$
(since $M^{2}=0$). From $T_{i}\eta_{E_{1}}=-\eta_{E_{1}}$ for $1\leq i<m-1$ it
follows that $\omega_{i}\eta_{E_{1}}=t^{i-m}\eta_{E_{1}}$. Thus $\omega
_{i}\eta_{E_{1}}=t^{c(i,E_{1})}\eta_{E_{1}}$ for $1\leq i\leq N$.
\end{proof}

\subsection{Steps}

Having found two polynomials which are $\{\omega_{i}\} $
simultaneous eigenfunctions we describe the method for constructing for each
$E\in\mathcal{Y}_{0}\cup\mathcal{Y}_{1}$ a polynomial $\tau_{E}\in
\mathcal{P}_{m}$ such that $\omega_{i}\tau_{E}=t^{c(i,E) }%
\tau_{E}$ for $1\leq i\leq N$. Recall the standard properties $T_{i}\omega
_{j}=\omega_{j}T_{i}$ for $i<j-1$ (obvious) and for $i>j$ (suppose $i=j+1$)
then%
\begin{align*}
T_{j+1}\omega_{j} & =t^{-2}T_{j+1}T_{j}T_{j+1}\omega_{j+2}T_{j+1}%
T_{j}=t^{-2}T_{j}T_{j+1}T_{j}\omega_{j+2}T_{j+1}T_{j}=t^{-2}T_{j}T_{j+1}\omega_{j+2}T_{j}T_{j+1}T_{j}
\\
&=t^{-2}T_{j}T_{j+1}%
\omega_{j+2}T_{j=1}T_{j}T_{j+1}=\omega_{j}T_{j+1},
\end{align*}
by the braid relations; and%
\begin{align}
T_{j}\omega_{j} & =t^{-1}T_{j}^{2}\omega_{j+1}T_{j}=t^{-1}\{(t-1) T_{j}+t\} \omega_{j+1}T_{j}
 =(t-1) \omega_{j}+\omega_{j+1}T_{j},\nonumber
\\
\omega_{j}T_{j} & =T_{j}\omega_{j+1}+(t-1) \omega_{j}.\label{Tomego}
\end{align}

\begin{Proposition}\label{Tf2g}
Suppose $\omega_{j}f=\lambda_{j}f$ for $1\leq j\leq N$ $(
f\neq0)$, $\lambda_{i}\neq\lambda_{i+1}$ and
\[
g:=T_{i}f+\dfrac{(t-1) \lambda_{i}}{\lambda_{i+1}-\lambda_{i}}f
\]
then $\omega
_{j}g=\lambda_{j}g$ for all $j\neq i,i+1$ and $\omega_{i}g=\lambda
_{i+1}g$, $\omega_{i+1}g=\lambda_{i}g$. If $\lambda_{i+1}\neq t^{\pm1}\lambda
_{i}$ then $g\neq0$.
\end{Proposition}

\begin{proof}
If $j>i+1$ or $j<i$ then $\omega_{j}T_{i}=T_{i}\omega_{j}$ and thus
$\omega_{j}g=\lambda_{j}g$. By (\ref{Tomego})%
\begin{align*}
\omega_{i}g & =\omega_{i}T_{i}f+\lambda_{i}\dfrac{(t-1)
\lambda_{i}}{\lambda_{i+1}-\lambda_{i}}f=T_{i}\omega_{i+1}f+\bigg\{
t-1+\dfrac{(t-1) \lambda_{i}}{\lambda_{i+1}-\lambda_{i}%
}\bigg\} \lambda_{i}f
\\
& =\lambda_{i+1}T_{i}f+\frac{(t-1) \lambda_{i+1}\lambda_{i}%
}{\lambda_{i+1}-\lambda_{i}}f=\lambda_{i+1}g.
\end{align*}
A similar calculation using $\omega_{i+1}T_{i}=T_{i}\omega_{i}-(t-1) \omega_{i}$ shows that $\omega_{i+1}g=\lambda_{i}g.$ Since $T_{i}^{2}=(t-1) T_{i}+t$%
\[
\bigg(T_{i}+\dfrac{(t-1) \lambda_{i+1}}{\lambda_{i}-\lambda_{i+1}}\bigg)
\bigg(T_{i}+\dfrac{(t-1) \lambda_{i}}{\lambda_{i+1}-\lambda_{i}}\bigg)
=\frac{(\lambda_{i}t-\lambda_{i+1})
(\lambda_{i}-t\lambda_{i+1})}{(\lambda_{i}-\lambda_{i+1})^{2}},
\]
thus $\lambda_{i+1}\neq t^{\pm1}\lambda_{i}$ implies $g\neq0$.
\end{proof}

Given the hypotheses of the proposition and the self-adjointness of
$\omega_{i}$ (Corollary~\ref{omadjt}) it follows that $\langle
f,g\rangle =0$ ($\lambda_{i}\langle f,g\rangle =\langle
\omega_{i}f,g\rangle =\langle f,\omega_{i}g\rangle
=\lambda_{i+1}\langle f,g\rangle $).

\begin{Lemma}\label{fgnorm}
Suppose $g=(T_{i}+b) f$ and $\langle f,g\rangle =0$ then $\Vert g\Vert^{2}=(1-b)
(t+b) \Vert f\Vert^{2}$.
\end{Lemma}

\begin{proof}
It follows from $T_{i}$ being self-adjoint that $\langle T_{i}f,T_{i}f\rangle
=\left\langle T_{i}^{2}f,f\right\rangle =(t-1) \langle T_{i}f,f\rangle +t\Vert f\Vert^{2}$
and $\langle f,g\rangle =0$ implies $\langle T_{i}f,f\rangle +b\Vert f\Vert^{2}=0$. Thus%
\begin{gather*}
\Vert g\Vert^{2} =\Vert T_{i}f\Vert^{2}
+2b\langle T_{i}f,f\rangle +b^{2}\Vert f\Vert^{2}
=(t-1+2b)\langle T_{i}f,f\rangle +(t+b^{2}) \Vert f\Vert^{2}
\\ \hphantom{\Vert g\Vert^{2} }
{} =\big\{{-}b(t-1+2b) +t+b^{2}\big\} \Vert f\Vert^{2}
=\big({-}b^{2}-b(t-1) +t\big) \Vert f\Vert^{2} =(1-b) (t+b) \Vert f\Vert^{2}.\!\!\!\!\!\!
\tag*{\qed}
\end{gather*}
\renewcommand{\qed}{}
\end{proof}

\begin{Corollary}\label{g2zer}
Suppose $g=(T_{i}+b) f$, $\langle g,f\rangle =0$ and $b=1$ or $b=-t$ then $g=0$.
\end{Corollary}

\subsubsection[Isotype (N-m,1m)]{Isotype $\boldsymbol{( N-m,1^{m})}$}

This concerns the polynomials in $\mathcal{P}_{m,0}=\ker D\cap\mathcal{P}%
_{m}=D\mathcal{P}_{m+1}.$ Recall $\mathcal{Y}_{0}=\{E\colon \#E=m+1$, $N\in
E\} $.

\begin{Definition}
For $\#E=m+1$ define $\psi_{E}=D\phi_{E}$.
\end{Definition}

The set $\big\{\psi_{E}\colon E\in\mathcal{Y}_{0}\big\}$ spans $\mathcal{P}_{m,0}$;
for suppose $N\notin E$ then $\theta_{N}D\phi_{E}$ is a linear
combination of~$\phi_{F}$ with $F\in\mathcal{Y}_{0}$
and $t^{1-N}D\theta_{N}D\phi_{E}=D\phi_{E}=\psi_{E}$.
The map $p\big(\theta_{1},\dots,\theta_{N}\big) \rightarrow p\big(\theta_{1},\dots,\theta
_{N-1},0\big)$ takes $\psi_{E}$ to $t^{N-1}\phi_{E\backslash\{N\}}$; thus $\dim\mathcal{P}_{m,0}=\binom{N-1}{m}$. The function
$\operatorname{inv}(E) $ provides a partial order~on~$\mathcal{Y}_{0}$.

\begin{Definition}
For $0\leq n\leq m(N-1-m) $ let%
\[
\mathcal{P}_{m,0}^{(n)}:=\operatorname{span}\big\{\psi_{E}
\colon E\in\mathcal{Y}_{0},\,\operatorname{inv}(E) \leq n\big\}.
\]

\end{Definition}

The extreme cases are $\operatorname{inv}(\{N-m,\dots,N\}) =0$
and $\operatorname{inv}(\{1,2,\dots,m,N\}) =m(N-1-m)$.

\begin{Lemma}
Suppose $E\in\mathcal{Y}_{0}$ and $\operatorname{inv}(E) =n$. If
$\{i,i+1\} \subset E$ then $T_{i}\psi_{E}=-\psi_{E}$, or if
$\{i,i+1\} \cap E=\varnothing$ then $T_{i}\psi_{E}=t\psi_{E}$. If
$(i,i+1) \in E\times E^{C}$ then $\operatorname{inv}(s_{i}E) =n-1$
and $T_{i}\psi_{E}=\psi_{s_{i}E}$. If $i(i,i+1) \in E^{C}\times E$
then $\operatorname{inv}(s_{i}E)=n+1$ and $T_{i}\psi_{E}=(t-1) \psi_{E}+t\psi_{s_{i}E}%
\in\mathcal{P}_{m,0}^{(n+1)}$.
That is, $T_{i}\mathcal{P}_{m,0}^{(n)}\subset\mathcal{P}_{m,0}^{(n+1)}$.
\end{Lemma}

\begin{proof}
The transformation rules follow from Definition~\ref{defTi} and $DT_{i}%
=T_{i}D$ (see Proposition~\ref{commMD}).
\end{proof}

\begin{Theorem}\label{Yzeroex}
Suppose for some $n$ and for each $E\in\mathcal{Y}_{0}$ with
$\operatorname{inv}(E) =n$ there is a~polynomial $\tau_{E}=t^{n}%
\psi_{E}+p_{E}$ with $p_{E}\in\mathcal{P}_{m,0}^{(n-1)}$ such
that $\omega_{i}\tau_{E}=t^{c(i,E) }\tau_{E}$ for all $i$ then
this property holds for~$n+1$.
\end{Theorem}

\begin{proof}
Suppose $E\in\mathcal{Y}_{0}$ with $\operatorname{inv}(E) =n+1$ then
for some $i$ it holds that $(i,i+1) \in E\times E^{C}$
(otherwise $E=E_{0}=\{N-m,N-m+1,\dots,N\}$ and $\operatorname{inv}(E_{0}) =0$),
then let $F=s_{i}E$ so that $\operatorname{inv}(F) =n$. Then $c(j,E) =c(j,F) $ for $j\neq
i,i+1$ and $c(i,E) =c(i+1,F) \leq-1$, $c(i+1,E) =c(i,F) \geq1$. By Proposition~\ref{Tf2g} let%
\[
\tau_{E}=\bigg(T_{i}+\dfrac{(t-1) t^{c(i,F)}}{t^{c(i+1,F)}-t^{c(i,F)}}\bigg) \tau_{F}.
\]
Then $\omega_{j}\tau_{E}=t^{c(j,E) }\tau_{E}$ for all $j$; and
$\tau_{E}=(T_{i}+b) \big(t^{n}\psi_{F}+p_{F}\big)
=t^{n+1}\psi_{E}+t^{n}(t-1+b) \psi_{F}+(T_{i}+b)
p_{E}$ (for a constant $b$). By the lemma $t^{n}(t-1+b) \psi_{F}+(T_{i}+b) p_{E}\in\mathcal{P}_{m,0}^{(n) }$.
\end{proof}

\begin{Corollary}
For each $E\in\mathcal{Y}_{0}$ and $\operatorname{inv}(E) =n$ there is
a unique $\tau_{E}\in\mathcal{P}_{m,0}$ with $\tau_{E}=t^{n}\psi_{E}+p_{E}$
and $p_{E}\in\mathcal{P}_{m,0}^{(n-1) }$ such that $\omega
_{i}\tau_{E}=t^{c(i,E) }\tau_{E}$ for all $i$.
\end{Corollary}

\begin{proof}
The existence follows from induction starting with $\tau_{E_{0}}=\psi_{E_{0}}$
and Theorem~\ref{psi0eigen}. Uniqueness follows from the leading term. The
$\{\omega_{i}\} $-eigenvalues of $\tau_{E}$ determine $E$ uniquely.
\end{proof}

\begin{Corollary}\label{onedim}
Suppose $E\in\mathcal{Y}_{0}$, if $\{i,i+1\}
\subset E$ then $T_{i}\tau_{E}=-\tau_{E}$ and if $\{i,i+1\} \cap
E=\varnothing$ or $i=N-1\notin E$ then $T_{i}\tau_{E}=t\tau_{E}$.
\end{Corollary}

\begin{proof}
Let $b=\frac{(t-1) t^{c(i,E) }}{t^{c(i+1,E)}-t^{c(i,E)}}
=\frac{t-1}{t^{c(i+1,E)-c(i,E)}-1}$. If $\{i,i+1\} \subset E$ then
$c(i+1,E) =1+c(i,E)$ and $b=1$; thus
$\left\langle (T_{i}+1) \tau_{E},\tau_{E}\right\rangle =0$ (by
the comment after Proposition~\ref{Tf2g}) and $\Vert (T_{i}+1)\tau_{E}\Vert^{2}$ $=0$ by Corollary~\ref{g2zer}. If~$\{i,i+1\} \cap E=\varnothing$ (or $i=N-1\notin E$) then $c(i+1,E) =c(i,E) -1$ and $b=\frac{t-1}{t^{-1}-1}=-t$; thus
$\left\Vert (T_{i}-t)\tau_{E}\right\Vert^{2}=0$.
\end{proof}

\begin{Definition}\label{defuz}
Let $u(z):=\frac{(t-z) (1-tz) }{(1-z)^{2}}$. Suppose $v\in\mathbb{Z}^{N}$ and
$v_{j}\neq0$ for $j<N$, $v_{N}=0$, then%
\begin{gather}\label{Cprod1}
\mathcal{C}(v) :=\prod\limits_{1\leq i<j<N}\big\{u\big(t^{v_{i}-v_{j}}\big) \colon v_{i}<0<v_{j}\big\}.
\end{gather}

\end{Definition}

\begin{Proposition}
\label{tau0norm}Suppose $E\in\mathcal{Y}_{0}$ then
\[
\Vert\tau_{E}\Vert^{2}=t^{2(N-m-1)}[m+1]_{t}\mathcal{C}\big([c(i,E)]_{i=1}^{N}\big).
\]

\end{Proposition}

\begin{proof}
By definition $\tau_{E_{0}}=\sum_{j=N-m}^{N}t^{j-1}(-1)
^{N-m-j}\phi_{E_{0}\backslash\{j\} }$ and $\left\Vert
\tau_{E_{0}}\right\Vert^{2}=\sum_{j=N-m}^{N}t^{2j-2}t^{-i_{j}}$, where
$i_{j}=\operatorname{inv}\big(E_{0}\backslash\{j\} \big)
=j-N+m$; thus $\Vert \tau_{E_{0}}\Vert^{2}=t^{2(N-m-1)}\sum_{j=0}^{m}t^{j}$. Suppose the formula is valid for all $E$
with $\operatorname{inv}(E) \leq n$ and $\operatorname{inv}(E)
=n+1$. Then $E=s_{i}F$ for some $i\in E$ with $i+1\notin E$ and $\operatorname{inv}%
(F) =n$ (so $i+1\in F$, $i\notin F$). By Lemma~\ref{fgnorm}
$\Vert\tau_{E}\Vert^{2}=(1-b) (t+b)
\Vert \tau_{F}\Vert^{2}$, where $b=\frac{t-1}{t^{c(i+1,F) -c(i,F)}-1}$. Write $z=t^{c(i+1,F)-c(i,F)}=t^{c(i,E) -c(i+1,E)}$ then%
\[
\Vert\tau_{E}\Vert^{2}=\frac{(t-z) (1-tz) }{(1-z)^{2}}\Vert \tau_{F}\Vert^{2}
=u(z) \Vert \tau_{F}\Vert^{2}.
\]
In~the product for $\Vert \tau_{F}\Vert^{2}$ the factors for pairs
$(i)$~$(k,\ell)$ with $\{k,\ell\} \cap\{i,i+1\} =\varnothing$, $(ii)$~$(k,i)$, $k\in F$, $k<i$, $(iii)$~$(i+1,\ell)$, $\ell\notin F,\ell>i+1$ are the same in the product
for $\Vert\tau_{E}\Vert^{2}$ for the pairs $(i)$~$(k,\ell)$, $(ii)$~$(k,i+1)$, $(iii)$~$(i,\ell)$, respectively. The extra factor in the product for $\Vert \tau_{E}\Vert^{2}$ has the desired value.
\end{proof}

\subsubsection[Isotype (N-m+1, 1 m-1)]{Isotype $\boldsymbol{(N-m+1,1^{m-1})}$}

This concerns the polynomials in $\mathcal{P}_{m,1}=\ker M\cap\mathcal{P}_{m}=M\mathcal{P}_{m-1}$.

\begin{Definition}
For $\#E=m-1$ define $\eta_{E}:=M\phi_{E}$. The set
$\mathcal{Y}_{1}=\big\{E\colon \#E=m-1,N\notin E\big\}$.
\end{Definition}

The set $\big\{\eta_{E}\colon E\in\mathcal{Y}_{1}\big\} $ spans $\mathcal{P}_{m,1}$;
for suppose $N\in E$ then $\partial_{N}M\phi_{E}$ is a linear
combination of $\phi_{F}$ with $F\in\mathcal{Y}_{1}$ and
$M\big(\partial_{N}M\phi_{E}\big) =M\phi_{E}=\eta_{E}$. Furthermore $\partial
_{N}\eta_{E}=\phi_{E}$ and thus $\dim\mathcal{P}_{m,1}=\binom{N-1}{m-1}$. The
function $\operatorname{inv}(E)$ provides a partial order on
$\mathcal{Y}_{1}$. The extreme cases are
$\operatorname{inv}(\{N-m+1,\dots,N-1\}) =m-1$ and
$\operatorname{inv}(\{1,2,\dots,m-1\}) =(m-1) (N-m+1) $.

\begin{Definition}
For $0\leq n\leq(m-1) (N-m+1) $ let%
\[
\mathcal{P}_{m,1}^{(n) }:=\operatorname{span}\big\{\eta_{E}%
\colon E\in\mathcal{Y}_{1},\,\operatorname{inv}(E) \geq n\big\} .
\]

\end{Definition}

\begin{Lemma}
Suppose $E\in\mathcal{Y}_{1}$ and $\operatorname{inv}(E) =n$. If
$\{i,i+1\} \subset E$ then $T_{i}\eta_{E}=-\eta_{E}$, or if
$\{i,i+1\} \cap E=\varnothing$ then $T_{i}\eta_{E}=t\eta_{E}$.
If~$i\notin E$, $i+1\in E$ then $\operatorname{inv}(s_{i}E) =n+1$ and
$T_{i}\eta_{E}=(t-1) \eta_{E}+t\eta_{s_{i}E}\in\mathcal{P}%
_{m,1}^{(n)}$. If~$i\in E$, $i+1\notin E$ then $\operatorname{inv}(
s_{i}E) =n-1$ and $T_{i}\eta_{E}=\eta_{s_{i}E}\in\mathcal{P}%
_{m,1}^{(n-1)}$.
\end{Lemma}

\begin{proof}
The transformation rules follow from Definition~\ref{defTi} and $MT_{i}%
=T_{i}M$.
\end{proof}

\begin{Theorem}
Suppose for some $n$ and for each $E\in\mathcal{Y}_{1}$ with $\operatorname{inv}%
(E) =n$ there is a polynomial $\tau_{E}=\eta_{E}+p_{E}$ with
$p_{E}\in\mathcal{P}_{m,1}^{(n+1) }$ such that $\omega_{i}%
\tau_{E}=t^{c(i,E) }\tau_{E}$ for all $i$ then this property
holds for~$n-1$.
\end{Theorem}

\begin{proof}
Suppose $E\in\mathcal{Y}_{1}$ with $\operatorname{inv}(E) =n-1$ then
for some $i$ it holds that $i\notin E,i+1\in E$ (otherwise $E=E_{1}=\{1,\dots,m-1\}$ and $\operatorname{inv}(E_{1}) =(m-1) (N-m+1)$), then let $F=s_{i}E$ so that
$\operatorname{inv}(F) =n$. Then $c(j,E) =c(j,F)$ for $j\neq i,i+1$ and $c(i,E) =c(i+1,F) \geq1$, $c(i+1,E) =c(i,F) \leq-1$. By~Proposition~\ref{Tf2g} let
\[
\tau_{E}=\bigg( T_{i}+\dfrac{(t-1) t^{c(i,F)}}{t^{c(i+1,F) }-t^{c(i,F) }}\bigg) \tau_{F}.
\]
Then $\omega_{j}\tau_{E}=t^{c(j,E) }\tau_{E}$ for all $j$; and
$\tau_{E}=(T_{i}+b) \big(\eta_{F}+p_{F}\big) =\eta
_{E}+b\eta_{F}+(T_{i}+b) p_{F}$ (for a~constant~$b$). By the
previous lemma $b\eta_{F}+(T_{i}+b) p_{E}\in\mathcal{P}_{m,1}^{(n)}$.
\end{proof}

\begin{Corollary}
For each $E\in\mathcal{Y}_{1}$ and $\operatorname{inv}(E)=n\leq(m-1) (N-m+1)$ there is a unique
$\tau_{E}\in\mathcal{P}_{m,1}$ with $\tau_{E}=\eta_{E}+p_{E}$ and
$p_{E}\in\mathcal{P}_{m,1}^{(n+1)}$ such that $\omega_{i}\tau_{E}=t^{c(i,E) }\tau_{E}$ for all $i$.
\end{Corollary}

\begin{proof}
The existence follows from induction starting with $\tau_{E_{1}}=M\phi_{E_{1}}=\eta_{E_{1}}$ and Theorem~\ref{eta0eigen}. Uniqueness follows from the
leading term. The eigenvalues of $\tau_{E}$ determine $E$ uniquely.
\end{proof}

\begin{Corollary}
Suppose $E\in\mathcal{Y}_{1}$, if $\{i,i+1\} \in E$ then
$T_{i}\tau_{E}=-\tau_{E}$ and if $\{i,i+1\} \cap E=\varnothing$
then $T_{i}\tau_{E}=t\tau_{E}$.
\end{Corollary}

\begin{proof}
This has the same proof as Corollary~\ref{onedim}.
\end{proof}

\begin{Proposition}\label{tauEnorm}
Suppose $E\in\mathcal{Y}_{1}$ then%
\[
\Vert\tau_{E}\Vert^{2}=t^{-m(N-m)}[N-m+1]_{t}~\mathcal{C}\big([-c(i,E)]_{i=1}^{N}\big).
\]

\end{Proposition}

\begin{proof}
By definition $\tau_{E_{1}}=(-1)^{m-1}\sum_{j=m}^{N}\phi_{E_{1}\cup\{j\} }$ and
$\Vert \tau_{E_{1}}\Vert^{2}=\sum_{j=m}^{N}t^{-i_{j}}$, where
$i_{j}=\operatorname{inv}\big(E_{1}\cup\{j\}\big) =m(N+1-m)-j$; thus
$\Vert \tau_{E_{1}}\Vert^{2}=t^{-m(N-m) }\sum_{j=0}^{N-m}t^{j}$.
Suppose the formula is valid for all $E$ with
$\operatorname{inv}(E) \geq n$ and $\operatorname{inv}(E)=n-1$. Then $E=s_{i}F$ for some $i\notin E$ with~$i+1\in E$ and $\operatorname{inv}(F) =n$ (so $(i,i+1) \in F\times F^{C}$). By Lemma~\ref{fgnorm} $\Vert\tau_{E}\Vert^{2}=(1-b) (t+b) \Vert \tau_{F}\Vert^{2}$, where $b=\frac{t-1}{t^{c(i+1,F) -c(i,F) }-1}$. Write $z=t^{c(i+1,F) -c(i,F)}=t^{c(i,E)-c(i+1,E)}$ then%
\[
\Vert\tau_{E}\Vert^{2}=\frac{(t-z) (1-tz) }{(1-z)^{2}}\Vert \tau_{F}\Vert
^{2}=u(z) \Vert \tau_{F}\Vert^{2}.
\]
In~the product for $\vert \tau_{F}\vert^{2}$ the factors for pairs
$(i)$~$(k,\ell)$ with $\{k,\ell\}\cap\{i,i+1\} =\varnothing$, $(ii)$~$(k,i)$, $k\notin F$, $k<i$,
$(iii)$~$(i+1,\ell)$, $\ell\in F$, $\ell>i+1$ are the same in the
product for $\vert \tau_{E}\vert^{2}$ for the pairs $(i)$~$(k,\ell)$, $(ii)$~$(k,i+1)$,
$(iii)$~$(i,\ell)$, respectively. The extra factor in the product for $\vert \tau_{E}\vert^{2}$ has the desired value.
\end{proof}

\begin{Proposition}\label{Dnorm}
Let $F_{0}=\{1,\dots,m,N\}$, $F_{1}=\{1,\dots,m\} $ then $\tau_{F_{0}}\in\mathcal{P}_{m,0}$
and $\tau_{F_{1}}\in\mathcal{P}_{m+1,1}$ have the same $\{\omega_{i}\}$-eigenvalues and%
\begin{gather*}
\Vert \tau_{F_{0}}\Vert^{2} =t^{(m+2) (N-m-1)}\frac{[N]_{t}}{[N-m]_{t}},\qquad
\Vert \tau_{F_{1}}\Vert^{2}=t^{-(m+1) (N-m-1)}[N-m]_{t},
\\
\Vert \tau_{F_{0}}\Vert^{2} =t^{(2m+3) (N-m-1)}[N]_{t}[N-m]_{t}^{-2}
\Vert\tau_{F_{1}}\Vert^{2},
\\
D\tau_{F_{1}} =t^{-(m+1) (N-m-1)}[N-m]_{t}\tau_{F_{0}},\qquad
\Vert D\tau_{F_{1}}\Vert^{2}=t^{N-m-1}[N]_{t}\Vert \tau_{F_{1}}\Vert^{2}.
\end{gather*}

\end{Proposition}

\begin{proof}
The content vector for $F_{0}$ is $[-m,1-m,\dots,-1,N-m-1,\dots,1,0]$ (the same as $F_{1}$) so that%
\begin{gather*}
\mathcal{C}\big([c(i,F_{0})]_{i=1}^{N}\big) =\prod\limits_{i=1}^{m}\prod\limits_{j=1}^{N-m-1}
\frac{(t-t^{-i-j})(1-t^{1-i-j})}{(1-t^{-i-j})^{2}}
\\ \phantom{\mathcal{C}\big([c(i,F_{0})]_{i=1}^{N}\big)}
{}=t^{m(N-m-1)}\prod\limits_{i=1}^{m}\prod\limits_{j=1}^{N-m-1}
\bigg(\frac{t^{i+j+1}-1}{t^{i+j}-1}\bigg) \bigg(\frac{t^{i+j-1}-1}{t^{i+j}-1}\bigg)
\\ \phantom{\mathcal{C}\big([c(i,F_{0})]_{i=1}^{N}\big)}
{}=t^{m(N-m-1)}\prod\limits_{j=1}^{N-m-1}
\bigg(\frac{t^{m+1+j}-1}{t^{1+j}-1}\bigg) \bigg(\frac{t^{j}-1}{t^{m+j}-1}\bigg)
\\ \phantom{\mathcal{C}\big([c(i,F_{0})]_{i=1}^{N}\big)}
{}=t^{m(N-m-1)}\bigg( \frac{t-1}{t^{N-m}-1}\bigg)
\bigg(\frac{t^{N}-1}{t^{m+1}-1}\bigg) =\frac{t^{m(N-m-1)}[N]_{t}}{[N-m]_{t}[m+1]_{t}}
\end{gather*}
by use of telescoping arguments. Thus $\Vert \tau_{F_{0}}\Vert^{2}=t^{(m+2) ( N-m-1) }\frac{[N]_{t}}{[N-m]_{t}}$ by Proposition~\ref{tau0norm}. The~value of
$\Vert \tau_{F_{1}}\Vert^{2}$ is from Proposition~\ref{tauEnorm}.
By definition $\tau_{F_{1}}=M\phi_{F_{1}}=(\theta_{m+1}+\cdots
+\theta_{N})\times \theta_{1}\theta_{2}\cdots\theta_{m}$ and the coefficient
of $\phi_{F_{1}}$ in $D\tau_{F_{1}}$ is $\sum_{i=m+1}^{N}t^{i-1}%
=t^{m}[N-m]_{t}$. The coefficient of~$\phi_{F_{1}}$ in
$\tau_{F_{0}}=D\phi_{F_{1}}$ is $(-1)^{m}t^{m(N-1-m) +N-1}$(see Theorem~\ref{Yzeroex}). From $D\tau_{F_{1}}$ and
$\tau_{F_{0}}$ having the same $\{\omega_{i}\} $-eigenvalues it
follows that $D\tau_{F_{1}}=a\tau_{F_{0}}$ for some constant.
\end{proof}

\subsection{Isomorphisms}

This section concerns the action of the maps $M$, $D$ on the irreducible
$\mathcal{H}_{N}(t)$-modules. The~fol\-lowing is a version of
Schur's lemma for irreducible representations.

\begin{Lemma}
Suppose $\mu$ is a linear isomorphism $V_{1}\rightarrow V_{2}$ of irreducible
$\mathcal{H}_{N}(t) $-modules such that $T_{i}\mu=\mu T_{i}$ for
$1\leq i<N$ and $V_{1}$, $V_{2}$ are equipped with inner products in which each
$T_{i}$ is self-adjoint, then $\Vert \mu f\Vert^{2}/\Vert f\Vert^{2}$ is constant for $f\in V_{1}$.
\end{Lemma}

\begin{proof}
The argument is based on orthogonal bases defined in the previous sections. By
hypo\-thesis $V_{1}$ has an orthogonal basis consisting of $\{\omega_{i}\} $-eigenfunctions.
The image of this basis under $\mu$ has the
same property. For a typical basis element $f\in V_{1}$ suppose $\omega
_{j}f=\lambda_{j}f$ for all~$j$ and $\lambda_{i+1}\neq t^{\pm1}\lambda_{i}$
then $g=(T_{i}+b) f$ satisfies $\omega_{i}g=\lambda_{i+1}g,\omega_{i+1}g=\lambda_{i}g$ for $b=\frac{(t-1)\lambda_{i}}{\lambda_{i+1}-\lambda_{i}}$ and $\Vert g\Vert
^{2}=(1-b) (t+b) \Vert f\Vert^{2}$
(this equation follows from $T_{i}$ being self-adjoint and $\langle
f,g\rangle =0$). By~hypo\-thesis $\omega_{j}\mu f=\lambda_{j}\mu f$ for
all $j$ and $\mu g=(T_{i}+b) \mu f$ satisfies $\omega_{i}\mu
g=\lambda_{i+1}\mu g$, $\omega_{i+1}\mu g=\lambda_{i}\mu g$. By Lemma~\ref{fgnorm} $\Vert \mu g\Vert^{2}=(1-b) (t+b) \Vert \mu f\Vert^{2}$ and so $\gamma:=\left\Vert \mu
g\right\Vert^{2}/\Vert g\Vert^{2}=\Vert \mu f\Vert
^{2}/\Vert f\Vert^{2}$. By the step constructions $\Vert \mu
f\Vert^{2}/\Vert f\Vert^{2}=\gamma$ holds for every basis
vector of $V_{1}$.
\end{proof}

The relation $MD+DM=[N]_{t}$ (Proposition~\ref{MDDM}) implies
that $\mathcal{P}_{m}$ is a direct sum of $\mathcal{P}_{m,0}=\mathcal{P}%
_{m}\cap\ker D$ and $\mathcal{P}_{m,1}=\mathcal{P}_{m}\cap\ker M$.

\begin{Theorem}
The maps $M$, $D$ are linear isomorphisms $\mathcal{P}_{m,0}\rightarrow
\mathcal{P}_{m+1,1}$, $\mathcal{P}_{m+1,1}\rightarrow\mathcal{P}_{m,0}$
res\-pec\-ti\-vely, of~$\mathcal{H}_{N}(t) $-modules, $\Vert
Mf\Vert^{2}=t^{m+1-N}[N]_{t}\Vert f\Vert
^{2}$ for $f\in\mathcal{P}_{m,0}$ and $\Vert Dg\Vert^{2}=t^{N-m-1}\times[N]_{t}\Vert g\Vert^{2}$ for
$g\in\mathcal{P}_{m+1,1}$.
\end{Theorem}

\begin{proof}
$M$ and $D$ commute with each $T_{i}$ and hence with each $\omega_{i}$.
Furthermore if $f\in\mathcal{P}_{m,0}$ then $(MD+DM)
f=DMf=[N]_{t}f$ (by Proposition~\ref{MDDM}) and thus $M$ is
one-to-one on~$\mathcal{P}_{m,0}$. Similarly if $g\in\mathcal{P}_{m+1,1}$ then
$[N]_{t}g=(MD+DM) g=MDg$ and $D$ is one-to-one.
By~the lemma there are constants $\gamma_{1},\gamma_{2}$ such that $\Vert
Mf\Vert^{2}=\gamma_{1}\Vert f\Vert^{2}$ and $\Vert
Dg\Vert^{2}=\gamma_{2}\Vert g\Vert^{2}$. From $(MD+DM) f=DMf$ it follows that $\Vert DMf\Vert^{2}=[N]_{t}^{2}\Vert f\Vert^{2}=\gamma_{2}\Vert
Mf\Vert^{2}$ and $\gamma_{1}=[N]_{t}^{2}/\gamma_{2}$. By~Proposition~\ref{Dnorm} $\gamma_{2}=t^{N-m-1}[N]_{t}$ and thus
$\gamma_{1}=t^{m+1-N}[N]_{t}$.
\end{proof}

\section{Nonsymmetric Macdonald polynomials}\label{NSMP}

\subsection{Operators on polynomials}

The following presents the key concepts for our constructions: the definition
of the action of~$\mathcal{H}_{N}(t)$ on superpolynomials and
the ingredients necessary to define the Cherednik operators whose simultaneous
eigenvectors are the nonsymmetric Macdonald superpolynomials. Here we extend
the polynomials in $\{\theta_{i}\}$ by adjoining $N$ commuting
variables $x_{1},\dots,x_{N}$ (that is \mbox{$[x_{i},x_{j}]=0$}, $[x_{i},\theta_{j}] =0$, $\theta_{i}\theta_{j}=-\theta_{j}\theta_{i}$ for all $i$,~$j$). Each polynomial is a sum of mono\-mi\-als~$x^{\alpha}\phi_{E}$, where $E\subset\{1,2,\dots,N\}$ and $\alpha
\in\mathbb{N}_{0}^{N}$, $x^{\alpha}:=\prod_{i=1}^{N}x_{i}^{\alpha_{i}}$.
The partitions in $\mathbb{N}_{0}^{N}$ are denoted by~$\mathbb{N}_{0}^{N,+}$
($\lambda\in\mathbb{N}_{0}^{N,+}$ if and only if $\lambda_{1}\geq\lambda
_{2}\geq\cdots\geq\lambda_{N}$). The \textit{fermionic} degree of this
monomial is $\#E$ and the \textit{bosonic} degree is $\vert
\alpha\vert :=\sum_{i=1}^{N}\alpha_{i}$. Let $s\mathcal{P}_{m}:=\operatorname{span}\big\{x^{\alpha}\phi_{E}\colon \alpha\in\mathbb{N}_{0}^{N},\,\#E=m\big\}$. Then using the decomposition $\mathcal{P}_{m}=\mathcal{P}_{m,0}\oplus\mathcal{P}_{m,1}$ let
\begin{gather*}
s\mathcal{P}_{m,0} =\operatorname{span}\big\{ x^{\alpha}\psi_{E}\colon \alpha
\in\mathbb{N}_{0}^{N},\,E\in\mathcal{Y}_{0}\big\},
\\
s\mathcal{P}_{m,1} =\operatorname{span}\big\{ x^{\alpha}\eta_{E}\colon \alpha
\in\mathbb{N}_{0}^{N},\,E\in\mathcal{Y}_{1}\big\}.
\end{gather*}
The Hecke algebra $\mathcal{H}_{N}(t) $ is represented on
$s\mathcal{P}_{m}$. This allows us to apply the theory of nonsymmetric
Macdonald polynomials taking values in $\mathcal{H}_{N}(t)$-modules (see~\cite{Dunkl2019,DunklLuque2012}).

\begin{Definition}
Suppose $p\in s\mathcal{P}_{m}$ and $1\leq i<N$ then set%
\[
\boldsymbol{T}_{i}p(x;\theta) :=(1-t) x_{i+1}
\frac{p(x;\theta) -p(xs_{i};\theta)}{x_{i}-x_{i+1}}+T_{i}p(xs_{i};\theta).
\]
\end{Definition}

Note that $T_{i}$ acts on the $\theta$ variables according to Definition~\ref{defTi}.

\begin{Theorem}[{\cite[Proposition~3.5]{DunklLuque2012}}]
Suppose $1\leq i<N-1$ then $\boldsymbol{T}_{i}\boldsymbol{T}_{i+1}\boldsymbol{T}_{i} =\boldsymbol{T}_{i+1}\boldsymbol{T}_{i}\boldsymbol{T}_{i+1}$, if~$1\leq i<N$ then $(\boldsymbol{T}_{i}+1) (\boldsymbol{T}_{i}-t) =0$ and if $1\leq
i<j-1\leq N-2$ then $\boldsymbol{T}_{i}\boldsymbol{T}_{j}=\boldsymbol{T}_{j}\boldsymbol{T}_{i}$.
\end{Theorem}

We also use%
\[
\boldsymbol{T}_{i}^{-1}p(x;\theta) =\bigg(\frac{1-t}{t}\bigg) x_{i}\frac{p(x;\theta) -p(xs_{i};\theta)}{x_{i}-x_{i+1}}+T_{i}^{-1}p(xs_{i};\theta).
\]

\begin{Definition}
Let $T^{(N) }=T_{N-1}T_{N-2}\cdots T_{1}$ and for $p\in
s\mathcal{P}_{m}$ and $1\leq i\leq N$%
\begin{gather*}
\boldsymbol{w}p(x;\theta) :=T^{(N)}p(qx_{N},x_{1},x_{2},\dots,x_{N-1};\theta),
\\
\xi_{i}p(x;\theta) :=t^{i-N}\boldsymbol{T}_{i}
\boldsymbol{T}_{i+1}\cdots\boldsymbol{T}_{N-1}\boldsymbol{wT}_{1}^{-1} \boldsymbol{T}_{2}^{-1}\cdots\boldsymbol{T}_{i-1}^{-1}p(x;\theta).
\end{gather*}
\end{Definition}

The operators $\xi_{i}$ are Cherednik operators, defined by Baker and
Forrester~\cite{BakerForrester97} (see Braverman et al.~\cite{Bravermanetal2020} for the
significance of these operators in double affine Hecke algebras). They
mutually commute (the proof in the vector-valued situation is in~\cite[Theorem~3.8]{DunklLuque2012}). Observe $\xi_{i-1}=t^{-1}\boldsymbol{T}_{i-1}\xi_{i}\boldsymbol{T}_{i-1}$. Their key properties are
\begin{gather}
T^{(N) }T_{i+1} =T_{N-1}\cdots T_{i+1}T_{i}T_{i+1}T_{i-1}\cdots T_{1}
 =T_{N-1}\cdots T_{i+2}T_{i}T_{i+1}T_{i}\cdots T_{1}=T_{i}T^{(N)},\nonumber
 \\
\big(T^{(N)}\big)^{2} =T^{(N)}\big(T_{N-1}\cdots T_{2}\big) T_{1}=\big(T_{N-2}\cdots T_{1}\big)
T^{(N)}T_{1}=T_{N-1}^{-1}\big(T^{(N)}\big)^{2}T_{1},\nonumber
\\
\big(T^{(N)}\big)^{2}T_{1} =T_{N-1}\big(T^{(N)}\big)^{2},\nonumber
\\
\boldsymbol{wT}_{i+1} =\boldsymbol{T}_{i}\boldsymbol{w},\qquad
\boldsymbol{w}^{2}\boldsymbol{T}_{1}=\boldsymbol{T}_{N-1}\boldsymbol{w}^{2},\nonumber
\\
\boldsymbol{w}^{-1}p(x;\theta) =T_{1}^{-1}T_{2}^{-1}\cdots
T_{N-1}^{-1}p( x_{2},x_{3},\dots,x_{N},q^{-1}x_{1};\theta).\label{TNprops}
\end{gather}
There is a basis of $s\mathcal{P}_{m}$ consisting of simultaneous eigenvectors
of $\{\xi_{i}\} $ and these are the nonsymmetric Macdonald
superpolynomials (henceforth abbreviated to \textquotedblleft
NSMP\textquotedblright).

Suppose $p(\theta) $ is independent of $x$ then $\boldsymbol{T}_{i}p=T_{i}p$ and
\begin{align*}
\xi_{i}p(\theta) & =t^{i-N}T_{i}T_{i+1}\cdots T_{N-1}\big(T_{N-1}\cdots T_{2}T_{1}\big) T_{1}^{-1}T_{2}^{-1}\cdots T_{i-1}^{-1}p(\theta)
\\
& =t^{i-N}T_{i}T_{i+1}\cdots T_{N-1}T_{N-1}\cdots T_{i}p(\theta) =\omega_{i}p(\theta) ,
\end{align*}
that is $\xi_{i}$ agrees with $\omega_{i}$ on polynomials of bosonic degree
0. Also $\boldsymbol{wT}_{i+1}=\boldsymbol{T}_{i}\boldsymbol{w}$. Suppose
$j>i+1$ then%
\begin{align*}
\xi_{i}\boldsymbol{T}_{j} & =t^{i-N}\boldsymbol{T}_{i}\boldsymbol{T}%
_{i+1}\cdots\boldsymbol{T}_{N-1}w\boldsymbol{T}_{j}\boldsymbol{T}_{1}%
^{-1}\boldsymbol{T}_{2}^{-1}\cdots\boldsymbol{T}_{i-1}^{-1}
\\
& =t^{i-N}\boldsymbol{T}_{i}\boldsymbol{T}_{i+1}\cdots\boldsymbol{T}%
_{N-1}\boldsymbol{T}_{j-1}w\boldsymbol{T}_{1}^{-1}\boldsymbol{T}_{2}%
^{-1}\cdots\boldsymbol{T}_{i-1}^{-1}
\\
& =t^{i-N}\boldsymbol{T}_{i}\cdots(\boldsymbol{T}_{j-1}\boldsymbol{T}%
_{j}\boldsymbol{T}_{j-1}) \cdots\boldsymbol{T}_{N-1}w\boldsymbol{T}%
_{1}^{-1}\boldsymbol{T}_{2}^{-1}\cdots\boldsymbol{T}_{i-1}^{-1}
\\
& =t^{i-N}\boldsymbol{T}_{i}\cdots( \boldsymbol{T}_{j}\boldsymbol{T}%
_{j-1}\boldsymbol{T}_{j}) \cdots\boldsymbol{T}_{N-1}w\boldsymbol{T}%
_{1}^{-1}\boldsymbol{T}_{2}^{-1}\cdots\boldsymbol{T}_{i-1}^{-1}=\boldsymbol{T}%
_{j}\xi_{i}.
\end{align*}
A similar argument shows $\xi_{i}\boldsymbol{T}_{j}=\boldsymbol{T}_{j}\xi_{i}$
when $j<i-1$, by using $\boldsymbol{T}_{j}^{-1}\boldsymbol{T}_{j+1}%
^{-1}\boldsymbol{T}_{j}=\boldsymbol{T}_{j+1}\boldsymbol{T}_{j}^{-1}%
\boldsymbol{T}_{j+1}^{-1}$.

\subsection{Properties of nonsymmetric Macdonald polynomials}

They have a triangularity property with respect to the partial order $\rhd$ on
the compositions $\mathbb{N}_{0}^{N}$, which is derived from the dominance
order:
\begin{gather*}
\alpha \prec\beta\Longleftrightarrow\sum_{j=1}^{i}\alpha_{j}\leq\sum_{j=1}^{i}\beta_{j},\qquad
1\leq i\leq N,\qquad \alpha\neq\beta,
\\
\alpha\lhd\beta \Longleftrightarrow(\vert \alpha\vert=\vert \beta\vert )
\wedge\big[\big(\alpha^{+}\prec\beta^{+}\big) \vee\big(\alpha^{+}
=\beta^{+}\wedge\alpha\prec\beta\big)\big].
\end{gather*}
The rank function on compositions is involved in the formula for an NSMP.

\begin{Definition}
For $\alpha\in\mathbb{N}_{0}^{N}$, $1\leq i\leq N$%
\[
r_{\alpha}(i) :=\#\{j\colon \alpha_{j}>\alpha_{i}\}
+\#\{j\colon 1\leq j\leq i,\,\alpha_{j}=\alpha_{i}\},
\]
then $r_{\alpha}\in\mathcal{S}_{N}$. There is a shortest expression
$r_{\alpha}=s_{i_{1}}s_{i_{2}}\cdots s_{i_{k}}$ and
$R_{\alpha}:=(T_{i_{1}}T_{i_{2}}\cdots T_{i_{k}})^{-1}\in\mathcal{H}_{N}(t)$ (that is, $R_{\alpha}=T(r_{\alpha})^{-1}$).
\end{Definition}

A consequence is that $r_{\alpha}\alpha=\alpha^{+}$, the \textit{nonincreasing
rearrangement} of $\alpha$, for any $\alpha\in\mathbb{N}_{0}^{N}$. For~example if $\alpha=(1,2,0,5,4,5)$ then $r_{\alpha}=[5,4,6,1,3,2]$ and
$r_{\alpha}\alpha=\alpha^{+}=(5,5,4,2,1,0)$ (recall $(u\alpha)_{i}=\alpha_{u^{-1}(i)}$).
Also $r_{\alpha}=I$ if and only if $\alpha\in\mathbb{N}_{0}^{N,+}$.

\begin{Theorem}[{\cite[Theorem~4.12]{DunklLuque2012}}] 
Suppose $\alpha\in\mathbb{N}_{0}^{N}$ and $E\in\mathcal{Y}_{k}$, $k=0,1$ then there exists a~$(\xi_{i})$-simultaneous eigenfunction%
\begin{gather}
M_{\alpha,E}(x;\theta) =t^{e(\alpha^{+})}q^{b(\alpha) }x^{\alpha}R_{\alpha}(\tau_{E}(\theta)) +\sum_{\beta\vartriangleleft\alpha}x^{\beta}v_{\alpha,\beta,E}(\theta;q,t), \label{MaEeqn}
\end{gather}
where $v_{\alpha,\beta,E}(\theta;q,t) \in\mathcal{P}_{m,k}$ and
its coefficients are rational functions of $q$, $t$. Also
$\xi_{i}M_{\alpha,E}(x;\theta) =\zeta_{\alpha,E}(i) M_{\alpha,E}(x;\theta)$, where $\zeta_{\alpha,E}(i)
=q^{\alpha_{i}}t^{c(r_{\alpha}(i),E)}$ for
$1\leq i\leq N$. The exponent $b(\alpha) :=\sum_{i=1}^{N}
\binom{\alpha_{i}}{2}$ and $e(\alpha^{+}) :=\sum_{i=1}^{N}\alpha_{i}^{+}(N-i+c(i,E))$.
\end{Theorem}

The \textit{spectral vector is} $[\zeta_{\alpha,E}(i)]_{i=1}^{N}$. Note that the leading term involves the element $R_{\alpha}(\tau_{E}(\theta))$ of $\mathcal{H}_{N}(t)$ acting on fermionic variables. The explanation for the exponents $e(\alpha^{+})$ and $b(\alpha)$ is in~Proposition~\ref{exponqt} below. The relations (\ref{Tomego}) hold when
$\omega_{i}$, $T_{i}$ is replaced by $\xi_{i}$, $\boldsymbol{T}_{i}$ respectively%
\begin{gather*}
\boldsymbol{T}_{j}\xi_{j} =(t-1) \xi_{j}+\xi_{j+1}\boldsymbol{T}_{j},
\\
\xi_{j}\boldsymbol{T}_{j} =\boldsymbol{T}_{j}\xi_{j+1}+(t-1)\xi_{j},
\end{gather*}
and this leads to the following, which has the same proof as Proposition~\ref{Tf2g}:

\begin{Proposition}\label{Mtrans}
Suppose $\xi_{j}f=\lambda_{j}f$ for $1\leq j\leq N$ $(f\neq0$ and $f\in s\mathcal{P}_{m})$ and $g:=\boldsymbol{T}_{i}f+\frac{t-1}{\lambda_{i+1}/\lambda_{i}-1}f$ then $\xi_{j}g=\lambda_{j}g$
for all $j\neq i,i+1$ and $\xi_{i}g=\lambda_{i+1}g$, $\xi_{i+1}g=\lambda_{i}g$.
If~$\lambda_{i+1}\neq t^{\pm1}\lambda_{i}$ then~$g\neq0$.
\end{Proposition}

This together with a degree-raising operation provides the method for
constructing the Macdonald polynomials.

Suppose $\alpha\in\mathbb{N}_{0}^{N}$, $E\in\mathcal{Y}_{0}\cup\mathcal{Y}_{1}$ and $\alpha_{i}\neq\alpha_{i+1}$ then let $z=\zeta_{\alpha,E}(i+1)/\zeta_{\alpha,E}(i)$ and
\begin{gather}
M_{s_{i}\alpha,E}=c_{\alpha,E}\bigg(\boldsymbol{T}_{i}+\frac{t-1}{z-1}\bigg) M_{\alpha,E},
\label{TiM1}
\end{gather}
where $c_{\alpha,E}=1$ if $\alpha_{i}<\alpha_{i+1}$ and $c_{\alpha,E}=u(z) $ if $\alpha_{i}>\alpha_{i+1}$. The spectral vector $\zeta_{s_{i}\alpha,E}=s_{i}\zeta_{\alpha,E}$.

Suppose $\alpha\in\mathbb{N}_{0}^{N}$ and $E\in\mathcal{Y}_{0}\cup
\mathcal{Y}_{1}$ and $\alpha_{i}=\alpha_{i+1}$, then let $j=r_{\alpha}(i)$.
If~$\{j,j+1\} \in E$ or~$j=N-1\in E$ then
$\boldsymbol{T}_{i}M_{\alpha,E}=-M_{\alpha,E}$. If $\{j,j+1\}
\cap E=\varnothing$ or $j=N-1\notin E$ then $\boldsymbol{T}_{i}M_{\alpha,E}=tM_{\alpha,E}$. If $j<N-2$ and $(j,j+1) \in E\times E^{C}$
or $(j,j+1) \in E^{C}\times E$ then $z=\zeta_{\alpha,E}(i+1) /\zeta_{\alpha,E}(i) =t^{c\left( j+1,E\right)
-c(j,E) }$ and%
\begin{gather*}
M_{\alpha,s_{j}E}=c_{\alpha,E}\bigg(\boldsymbol{T}_{i}+\frac{t-1}{z-1}\bigg) M_{\alpha,E}, \label{TiM2}%
\end{gather*}
where $c_{\alpha,E}=1$ if (1) $E\in\mathcal{Y}_{0}$ and $(j,j+1)
\in E^{C}\times E$ or if (2) $E\in\mathcal{Y}_{1}$ and $(j,j+1)
\in E\times E^{C}$, or $c_{\alpha,E}=u(z)^{-1}$ if (1)
$E\in\mathcal{Y}_{0}$ and $(j,j+1) \in E\times E^{C}$ or if (2)
$E\in\mathcal{Y}_{1}$ and $(j,j+1) \in E^{C}\times E$. In~all
cases $\zeta_{\alpha,s_{j}E}=s_{i}\zeta_{\alpha,E}.$

The above equations are implicit formulas for $\boldsymbol{T}_{i}M_{\alpha,E}$. Formula~\eqref{TiM1} is the same as that for the scalar case, as in~\cite{BakerForrester97, MimmachiNoumi98}.

The \textit{affine} step is defined as follows: for $\alpha\in\mathbb{N}_{0}^{N}$, $E\in\mathcal{Y}_{0}\cup\mathcal{Y}_{1}$%
\begin{gather*}
\Phi\alpha =(\alpha_{2},\alpha_{3},\dots,\alpha_{N},\alpha_{1}+1),
\\[.5ex]
\zeta_{\Phi\alpha,E} =\big[\zeta_{\alpha,E}(2),\zeta_{\alpha,E}(3),\dots,\zeta_{\alpha,E}(N)
,q\zeta_{\alpha,E}(1) \big],
\\[.5ex]
M_{\Phi\alpha,E} =x_{N}\boldsymbol{w}M_{\alpha,E}.
\end{gather*}
This is based on the relations $\xi_{N}x_{N}\boldsymbol{w}=qx_{N}%
\boldsymbol{w}\xi_{1}$ and $\xi_{i}x_{N}\boldsymbol{w}=x_{N}\boldsymbol{w}%
\xi_{i+1}$ for $1\leq i<N$.

Denote $\boldsymbol{0}=(0,\dots,0) \in\mathbb{N}_{0}^{N}$ and
recall $E_{0}=\{N-m,\dots,N\}$, $E_{1}=\{1,2,\dots,m-1\}$.
In~\cite[Section~4.1]{DunklLuque2012} a Yang--Baxter directed graph
method is used to inductively construct the $M_{\alpha,E}$ (this technique is
due to Lascoux~\cite{Lascoux2001}). Label the nodes $(\alpha,E) $; for
$\mathcal{Y}_{0}$ the root is $\left( \boldsymbol{0,}E_{0}\right) $ with
$M_{\boldsymbol{0,}E_{0}}=\tau_{E_{0}}=D\phi_{E_{0}}$; and for $\mathcal{Y}%
_{1}$ the root is $\left( \boldsymbol{0},E_{1}\right) $ with
$M_{\boldsymbol{0,}E_{1}}=\tau_{E_{1}}=M\phi_{E_{1}}$. The equations have been
set up so that $M_{\boldsymbol{0},E}=\tau_{E}$. The arrows in the graph point
from $(\alpha,E) $ to $(\Phi\alpha,E)$, and from
$(\alpha,E) $ to $(s_{i}\alpha,E)$ ($\alpha_{i}<\alpha_{i+1}$) or to $(\alpha,s_{j}E)$ when $j=r_{\alpha}(E)$ and $c_{\alpha,E}=1$ in the cases described above.

Here is a brief discussion of the effect of $\boldsymbol{T}_{i}$ on
$x^{\alpha}R_{\alpha}\tau_{E}$ for the $c_{\alpha,E}=1$ cases. For~$\alpha
\in\mathbb{N}_{0}^{N}$~let%
\[
\operatorname{inv}(\alpha) :=\#\big\{(i,j)\colon i<j,\,\alpha_{i}<\alpha_{j}\big\}
\]
then $r_{\alpha}\alpha=\alpha^{+}$ and $r_{\alpha}=s_{i_{1}}\cdots s_{i_{\ell}}$, where $\ell=\operatorname{inv}(\alpha)$. Recall $R_{\alpha
}=T_{i_{\ell}}^{-1}\cdots T_{i_{1}}^{-1}$ and the value of $R_{\alpha}$ is
independent of the expressions for $r_{\alpha}$ of length $\ell$. Suppose
$\alpha_{i}<\alpha_{i+1}$ then $\operatorname{inv}(\alpha)
=\operatorname{inv}(s_{i}\alpha)+1$; write $s_{i}\alpha
=r_{s_{i}\alpha}^{-1}\alpha_{+}$ and $r_{s_{i}\alpha}=s_{i_{1}}\cdots
s_{i_{\ell}}$ with $\ell=\operatorname{inv}(s_{i}\alpha)$. Thus
$r_{\alpha}^{-1}=s_{i}s_{i_{1}}\cdots s_{i_{\ell}}$ and $R_{\alpha}=T_{i}%
^{-1}R_{s_{i}\alpha}$ and so $\boldsymbol{T}_{i}x^{\alpha}R_{\alpha}\tau
_{E}=x^{s_{i}\alpha}R_{s_{i}\alpha}\tau_{E}+p(x;\theta) $, where
$p$ is a~sum of~$x^{\beta}p^{\prime}(\theta) $ with $s_{i}%
\alpha\vartriangleright\beta$.

{\sloppy
\begin{Proposition}
Suppose $\alpha_{i}=\alpha_{i+1}$ with $j=r_{\alpha}(i) $ then
$\boldsymbol{T}_{i}x^{\alpha}R_{\alpha}\tau_{E}=x^{\alpha}T_{i}R_{\alpha}%
\tau_{E}$ and \mbox{$T_{i}R_{\alpha}=R_{\alpha}T_{j}$}.
\end{Proposition}

}

\begin{proof}
Let $r_{\alpha}=s_{i_{1}}\cdots s_{i_{\ell}}$ with $\ell=\operatorname{inv}(\alpha) $. Let $\beta=\alpha$ except $\beta_{i+1}=\alpha_{i}+\frac
{1}{2}$ then $\operatorname{inv}(\beta) =\ell+1.$ From $r_{\alpha
}\alpha=\alpha^{+}$ it follows that $r_{\alpha}s_{i}\beta=\beta^{+}$. By
definition of $j$ we obtain $\beta_{j}^{+}=\alpha_{i}+\frac{1}{2},\beta
_{j+1}^{+}=\alpha_{i}$ and $r_{\alpha}^{-1}s_{j}\beta^{+}=\beta$. Now
$r_{\alpha}s_{i}\beta=\beta^{+}$ and $s_{j}r_{\alpha}\beta=\beta^{+}$ and
$r_{\alpha}s_{i},s_{j}r_{\alpha}$ have (at least) $\ell+1$ factors $s_{i_{n}}$
and $\operatorname{inv}(\beta) =\ell+1$. This implies $r_{\alpha}%
s_{i}=s_{j}r_{\alpha}$ and $T_{i_{1}}T_{i_{2}}\cdots T_{i_{\ell}}T_{i}%
=T_{j}T_{i_{1}}T_{i_{2}}\cdots T_{i_{\ell}}$, that is, $R_{\alpha}^{-1}%
T_{i}=T_{j}R_{\alpha}^{-1}$.
\end{proof}

Thus in the case $\alpha_{i}=\alpha_{i+1}$ the transformation laws (Theorem~\ref{Yzeroex} and Corollary~\ref{onedim}) apply to $\big(T_{j}+\frac
{t-1}{z-1}\big) \tau_{E}$ with $z=t^{c(j+1,E) -c(j,E)}$ (with the possible results $-\tau_{E},t\tau_{E},\tau_{s_{j}E},\dots$).

Hence only the affine step changes the power of $t$ in the coefficient of
$x^{\alpha}$. It remains to consider $x^{N}\boldsymbol{w}x^{\alpha}R_{\alpha
}\tau_{E}$.

\begin{Proposition}
Suppose $\alpha\in\mathbb{N}_{0}^{N}$ and $r_{\alpha}(1) =j$
then $T^{(N) }R_{\alpha}=t^{N-j}R_{\Phi\alpha}\omega_{j}$.
\end{Proposition}

\begin{proof}
Let $\ell=\#\{ (i,j) \colon 1<i<j\leq N,\,\alpha_{i}<\alpha_{j}\}$. Then there is a product $u=s_{i_{1}}s_{i_{2}}\cdots s_{i_{\ell}}$ such that $u\alpha=\big(\alpha_{1},\alpha_{1}^{+},\dots, \alpha_{j-1}^{+},\alpha_{j+1}^{+},\dots,\alpha_{N}^{+}\big)$, and
each $i_{k}>1$. By definition $i>j$ implies $\alpha_{1}^{+}\leq\alpha_{i}<\alpha_{1}+1$, and $i<j$ implies $\alpha_{i}^{+}\geq\alpha_{1}+1$. Then
$s_{j-1}s_{j-2}\cdots s_{1}u\alpha=\alpha^{+}$ and $R_{\alpha}^{-1}=\big(
T_{j-1}\cdots T_{1}\big)U$, where $U=T_{i_{1}}\cdots T_{i_{\ell}}$. Let
$u^{\prime}=s_{i_{1}-1}\cdots s_{i_{\ell}-1}$ then%
\[
u^{\prime}(\Phi\alpha) =\big(\alpha_{1}^{+},\dots,\alpha_{j-1}^{+},\alpha_{j+1}^{+}, \dots,\alpha_{N}^{+},\alpha_{1}+1\big)
\]
and $s_{j}s_{j+1}\cdots s_{N-1}u^{\prime}(\Phi\alpha) =(\Phi\alpha)^{+}$. Thus $R_{\Phi\alpha}^{-1}=\big(T_{j}\cdots T_{N-1}\big) U^{\prime}$ and $U^{\prime}=T_{i_{1}-1}\cdots T_{i_{\ell}-1}$.
By~(\ref{TNprops}) $T^{(N) }U=U^{\prime}T^{(N)}$ and
\begin{align*}
R_{\Phi\alpha}^{-1}T^{(N)} & =\big(T_{j}\cdots T_{N-1}\big) U^{\prime}T^{(N) }
=\big(T_{j}\cdots T_{N-1}\big) T^{(N) }U
\\
& =\big(T_{j}\cdots T_{N-1}T_{N-1}\cdots T_{j}\big) T_{j-1}\cdots
T_{1}U=t^{N-j}\omega_{j}R_{\alpha}^{-1}.
\tag*{\qed}
\end{align*}
\renewcommand{\qed}{}
\end{proof}

As a consequence $x_{N}\boldsymbol{w}x^{\alpha}R_{\alpha}\tau_{E}%
=q^{\alpha_{1}}t^{N-j+c(j,E) }x^{\Phi\alpha}R_{\Phi\alpha}%
\tau_{E}$ with $j=r_{\alpha}(1)$. Denote $(\alpha_{1}+1$, $\alpha_{2}+1,\cdots,\alpha_{N}+1) $ by $\alpha+\boldsymbol{1}$.

\begin{Corollary}
$(x_{N}\boldsymbol{w})^{N}M_{\alpha,E}=M_{\alpha+\boldsymbol{1},E}=q^{\vert \alpha\vert} t^{v}(x_{1}x_{2}\cdots x_{N}) M_{\alpha,E}$, where $v=\frac{N(N-1)}{2}+\sum_{i=1}^{N}c(i,E)$.
\end{Corollary}

\begin{proof}
Suppose the coefficient of $x^{\alpha}R_{\alpha}\left( \tau_{E}(\theta) \right) $ in $M_{\alpha,E}$ is $q^{a}t^{b}$ then%
\[
x_{N}\boldsymbol{w}M_{\alpha,E}=M_{\Phi\alpha,E}=q^{a+\alpha_{1}}t^{N-r_{\alpha}(1)
+c(r_{\alpha}(1),E)}x^{\Phi\alpha}R_{\Phi\alpha}(\tau_{E}) +\cdots.
\]
Then $(x_{N}\boldsymbol{w})^{2}M_{\alpha,E}$ involves
$r_{\Phi\alpha}(1) =r_{\alpha}(2)$. Repeating
this process yields a factor of $q^{\vert \alpha\vert}$ and the
$t$-exponent $\sum_{i=1}^{N}(N-r_{\alpha}(i) +c(r_{\alpha}(i),E)) =\frac{N(N-1)
}{2}+\sum_{i=1}^{N}c(i,E) .$ Furthermore $\Phi^{N}\alpha
=\alpha+\boldsymbol{1}$ and $R_{\alpha+\boldsymbol{1}}=R_{\alpha}$.
\end{proof}

If $E\in\mathcal{Y}_{0},\mathcal{Y}_{1}$ then $\sum_{i=1}^{N}c(i,E) =\frac{N(N-1)}{2}-Nm$, $\frac{N(N-1)}{2}-N(m-1) $ respectively (so $v=N(N-m-1)$ or~$N(N-m) $).

\begin{Proposition}
\label{exponqt}The exponent on $q$ in~\eqref{MaEeqn} is $b(\alpha) =\sum_{i=1}^{N}\binom{\alpha_{i}}{2}$. The exponent on $t$ in~\eqref{MaEeqn} is $\sum_{i=1}^{N}\lambda_{i}(N-i+c(i,E))$.
\end{Proposition}

\begin{proof}
For the value of $b(\alpha) $ we use induction on $\vert
\alpha\vert$. The statement $b(\boldsymbol{0}) =0$ is
true since $M_{\boldsymbol{0},E}=\tau_{E}$. The steps at $\alpha_{i}<\alpha_{i+1}$ taking $x^{\alpha}$ to $x^{s_{i}\alpha}$ do not involve $q$
(which does affect the other terms of the polynomial), and indeed $b(\alpha) $ is invariant under $s_{i}$. The affine step takes~$x^{\alpha}$ to $x^{\Phi\alpha}$ and multiplies by $q^{\alpha_{1}}$, that is, $b(\Phi\alpha) =b(\alpha) +\alpha_{1}$, and $\binom
{\alpha_{1}}{2}+\alpha_{1}=\binom{\alpha_{1}+1}{2}$. Induct on~$\ell(\lambda) $ for the $t$-exponent. Note that only the affine step affects
the exponent so it depends only on~$\lambda=\alpha^{+}$. Suppose $\lambda
_{k}\geq1$ and $\lambda_{i}=0$ for $i>k$. The affine step from $\lambda
^{(k) }$ with $\lambda_{i}^{(k) }=\lambda_{i}$
except $\lambda_{k}^{(k) }=\lambda_{k}-1$ proceeds by
$\lambda^{(k) }\rightarrow(\lambda_{k}-1,\lambda_{1},\dots,\lambda_{k-1},0,\dots,0) \rightarrow(\lambda_{1},\dots,\lambda_{k=1},0,\dots,0,\lambda_{k}) \rightarrow\lambda$
and multiplies by $t^{N-k+c(k,E) }$. Thus passing from
$(\lambda_{1},\dots,\lambda_{k-1},0,\dots,0)$ to $\lambda$ contributes
a~factor of $t^{\lambda_{k}(N-k+c(k,E))}$.
\end{proof}

\begin{Example}
Explicit formulas for $M_{\alpha,E}$ tend to be complicated; here is a fairly
simple one. Let $N=5$, $m=2$, $E=\{3,4,5\}$ and $\alpha=(0,0,1,0,0)$:
\begin{gather*}
M_{\alpha,E}=t^{6}x_{3}(t^{3}\theta_{2}\theta_{4}-t^{2}\theta_{2}\theta_{4}+\theta_{3}\theta_{4})
\\ \hphantom{M_{\alpha,E}=}
{}+\frac{(t-1)t^{9}q}{qt^{3}-1}\big\{x_{4}\big(t^{3}\theta_{2}\theta_{3}-t\theta_{2}\theta_{5} +\theta_{3}\theta_{5}\big)-x_{5}\big( t^{2}\theta_{2}\theta_{3}-t\theta_{2}\theta_{4}+\theta_{3}\theta_{4}\big) \big\} .
\end{gather*}
The spectral vector is $\big[ t,t^{-2},qt^{2},t^{-1},1\big] $ and
$\boldsymbol{T}_{4}M_{\alpha,E}=-M_{\alpha,E}$.
\end{Example}

There is a pairwise commuting set of (bosonic) degree lowering operations,
namely the Dunkl operators defined by Baker and Forrester~\cite{BakerForrester97}.

\begin{Definition}\label{defD}
Suppose $f\in s\mathcal{P}_{m}$ then $\boldsymbol{D}_{N}%
f:=\frac{1}{x_{N}}( f-\xi_{N}f) $ and if $i<N$ then
$\boldsymbol{D}_{i}f=\frac{1}{t}\boldsymbol{T}_{i}\boldsymbol{D}_{i+1}\boldsymbol{T}_{i}f$.
\end{Definition}

Assuming the existence of the nonsymmetric Macdonald polynomials $M_{\alpha
,E}$ the argument for showing that $\boldsymbol{D}_{i}f$ is a polynomial is
the following:

\begin{Proposition}
Suppose $\alpha\in\mathbb{N}_{0}^{N}$ and $E\in\mathcal{Y}_{0}\cup
\mathcal{Y}_{1}$; if $\alpha_{N}=0$ then $\boldsymbol{D}_{N}M_{\alpha,E}=0$ or
if $\alpha_{N}\geq1$ then $\boldsymbol{D}_{N}M_{\alpha,E}=(
1-\zeta_{\alpha,E}(N)) \boldsymbol{w}M_{\beta,E}$, where
$\alpha=\Phi\beta$. Also $\boldsymbol{D}_{N}M_{\Phi\alpha,E}=(
1-q\zeta_{\alpha}(1)) \boldsymbol{w}M_{\alpha,E}$.
\end{Proposition}

\begin{proof}
If $\alpha_{N}=0$ then $r_{\alpha}(N) =N,c(N,E)
=0$ and $\xi_{N}M_{\alpha,E}=M_{\alpha,E}$ and $(1-\xi_{N})
M_{\alpha,E}=0$. If $\alpha_{N}\geq1$ then $\alpha=\Phi\beta$ with $\vert
\beta\vert =\vert\alpha\vert -1$ and $(1-\xi_{N}) M_{\alpha,E}=(1-\zeta_{\alpha,E}(N))
M_{\alpha,E}=(1-\zeta_{\alpha,E}(N))x_{N}\boldsymbol{w}M_{\beta,E}$, and thus $\boldsymbol{D}_{N}M_{\alpha,E}=(1-\zeta_{\alpha,E}(N)) \boldsymbol{w}M_{\beta,E}$. For the other statement note that $\zeta_{\Phi\alpha,E}(N) =q\zeta_{\alpha,E}(1)$.
\end{proof}

\begin{Remark}\label{conjval}
We conjecture that there are evaluation formulas for the
special case $E=\{1,2,\dots,m,N\}$, $\alpha_{i}=0$ for $i>m$
and $x_{0}=\big(t^{N-1},t^{N-2},\dots,t,1\big)$. Here are some small
examples in isotype $(2,1,1)$%
\begin{gather*}
J_{(1,1,0,0),E}(x_{0}) =\frac{t^{5}\big(q-t^{4}\big)}{q-t^{2}}g_{12}(\theta),
\\
J_{(1,2,0,0),E}(x_{0}) =\frac{t^{7}\big(q-t^{4}\big) \big(q^{2}-t^{4}\big)}
{\big(q-t^{2}\big) \big(q^{2}-t^{3}\big)}g_{12}(\theta),
\\
J_{(2,1,0,0),E}(x_{0}) =\frac{t^{7}\big(q-t^{2}\big) \big(q-t^{4}\big) \big(q^{2}-t^{4}\big)}
{(q-t)\big(q-t^{2}\big) \big(q^{2}-t^{3}\big)}g_{12}(\theta),
\\
g_{12}(\theta) =\theta_{1}\theta_{2}-\dfrac{1}{t^{2}(1+t)}
(t\theta_{1}-\theta_{2})(\theta_{3}+\theta_{4}).
\end{gather*}
Replace $g_{12}(\theta)$ by $(\theta_{1}\theta_{2}\theta_{3}+\theta_{1}\theta_{2}\theta_{4})$ for the $\mathcal{P}_{3,1}$ version (by applying $M$).
\end{Remark}

\subsection{\label{SymFM}Symmetric bilinear form}

In~this section we define an inner product (symmetric bilinear form) on
$s\mathcal{P}_{m}$ in which $\boldsymbol{T}_{i}$, $\xi_{i}$ are self-adjoint, the
Macdonald polynomials are pairwise orthogonal and it is positive-definite for~$t,q>0$, $q\neq1$ and $\min\big(q^{1/N},q^{-1/N}\big) <t<\max\big(q^{1/N},q^{-1/N}\big)$. The background and proofs for this section are in~\cite{Dunkl2019}. The hypotheses $\langle \boldsymbol{T}_{i}f,g\rangle
=\langle f,\boldsymbol{T}_{i}g\rangle $ for $1\leq i< N$ and
$\langle \xi_{N}f,g\rangle =\langle f,\xi_{N}g\rangle $
already imply that $\langle \xi_{i}f,g\rangle =\langle
f,\xi_{i}g\rangle $ for all $i$ since $\xi_{i}=t^{-1}\boldsymbol{T}%
_{i}\xi_{i+1}\boldsymbol{T}_{i}$ and~thus $\langle M_{\alpha,E}%
,M_{\beta,F}\rangle =0$ if $(\alpha,E) \neq(
\beta,F) $ (at least one different $\{\xi_{i}\}
$-eigenvalue). Denote $\langle f,f\rangle =\Vert f\Vert
^{2}$, even if possibly nonpositive. The aim is to determine a formula for
$\Vert M_{\alpha,E}\Vert^{2}$ which, other than leading
coefficients $q^{\ast}t^{\ast}$, involves only linear factors of the form
$( 1-q^{q}t^{b}) $ with $a\in\mathbb{N}_{0}$, $b\in\mathbb{Z}$,
$\vert b\vert \leq N$. Recall $u(z) :=\frac{(t-z) (1-tz) }{(1-z)^{2}}$ from Definition~\ref{defuz}.

\begin{Proposition}
Suppose there is a symmetric bilinear form on $s\mathcal{P}$ in which each
$\boldsymbol{T}_{i}$ and $\xi_{i}$ is self-adjoint and suppose $E\in
\mathcal{Y}_{0}\cup\mathcal{Y}_{1}$, $\alpha\in\mathbb{N}_{0}^{N}$ and
$\alpha_{i}<\alpha_{i+1}$ for some $i$ then%
\[
\Vert M_{s_{i}\alpha,E}\Vert^{2}=u\big(q^{\alpha_{i+1}
-\alpha_{i}}t^{c(r_{\alpha}(i+1),E) -c(r_{\alpha}(i),E)}\big) \Vert M_{\alpha,E}\Vert^{2}.
\]
\end{Proposition}

\begin{proof}
This is the same argument used in Lemma~\ref{fgnorm}.
\end{proof}

We introduce a product for expressing $\Vert M_{\alpha,E}\Vert^{2}$ in terms
of~$\Vert M_{\alpha^{+},E}\Vert^{2}$.

\begin{Definition}\label{Rdef}
For $E\in\mathcal{Y}_{0}\cup\mathcal{Y}_{1}$, $\alpha\in\mathbb{N}_{0}^{N}$ let
\[
\mathcal{R}(\alpha,E) :=\prod\limits_{1\leq i<j\leq N,\,\alpha_{i}<\alpha_{j}}
u\big(q^{\alpha_{j}-\alpha_{i}}t^{c(r_{\alpha}(j),E) -c(r_{\alpha}(i),E)}\big) .
\]
\end{Definition}

There are $\operatorname{inv}(\alpha) $ terms in the product. The next
proposition assumes the same hypotheses on the bilinear form.

\begin{Proposition}\label{Mnorm1}
Suppose $E\in\mathcal{Y}_{0}\cup\mathcal{Y}_{1}$, $\alpha\in\mathbb{N}_{0}^{N}$ then
\[
\Vert M_{\alpha^{+},E}\Vert^{2}=\mathcal{R}(\alpha,E) \Vert M_{\alpha,E}\Vert^{2}.
\]
\end{Proposition}

\begin{proof}
With the same argument as in Proposition~\ref{tau0norm} one shows $\alpha_{i}<\alpha_{i+1}$ implies%
\begin{gather*}
\frac{\mathcal{R}(\alpha,E)}{\mathcal{R}(s_{i}\alpha,E)}
=u\big(q^{\alpha_{i+1}-\alpha_{i}}t^{c(r_{\alpha}(i+1),E) -c(r_{\alpha}(i),E)}\big).
\tag*{\qed}
\end{gather*}
\renewcommand{\qed}{}
\end{proof}

Another hypothesis is required to define the inner product for all polynomials
starting with bosonic degree $0$ ($M_{\boldsymbol{0},E}=\tau_{E}$). The
approach of making $\boldsymbol{D}_{i}$ the adjoint of multiplication by
$x_{i}$, or making an isometry out of the latter (torus norm) as is done in
the Jack polynomial situation, does not work here without a modification.

\begin{Theorem}[{\cite[Section~3.3]{Dunkl2019}}]
There is a unique symmetric bilinear form on
$s\mathcal{P}_{m}$ which extends the form in Definition~$\ref{phiform}$ and
satisfies $($for $f,g\in s\mathcal{P}_{m}$ and $1\leq i<N)$%
\begin{gather}
\langle \boldsymbol{T}_{i}f,g\rangle =\langle f,\boldsymbol{T}_{i}g\rangle ,\label{form1}
\\
\langle \xi_{N}f,g\rangle =\langle f,\xi_{N}g\rangle,\label{form2}
\\
\langle \boldsymbol{w}^{-1}\boldsymbol{D}_{N}f,g\rangle
=(1-q) \langle f,x_{N}\boldsymbol{w}g\rangle.\label{form3}%
\end{gather}
\end{Theorem}

It follows from $\xi_{i}=t^{-1}\boldsymbol{T}_{i}\xi_{i+1}\boldsymbol{T}_{i}$
that $\langle \xi_{i}f,g\rangle =\langle f,\xi_{i}g\rangle $ for all $i$. The reason for the fac\-tor~$(1-q)$ is to allow the limit as $t\rightarrow1$ when $q=t^{1/\kappa}$ to obtain
nonsymmetric Jack polynomials.

In~\cite{Dunkl2019} hypothesis (\ref{form3}) is stated in the equivalent form%
\[
\langle \boldsymbol{D}_{N}f,g\rangle =(1-q)\langle f,x_{N}\boldsymbol{ww}^{\ast}g\rangle,
\]
where
\begin{gather*}
\boldsymbol{w}^{\ast} =\boldsymbol{T}_{1}\boldsymbol{T}_{2}\cdots
\boldsymbol{T}_{N-1}\boldsymbol{wT}_{1}^{-1}\boldsymbol{T}_{2}^{-1}%
\cdots\boldsymbol{T}_{N-1}^{-1},
\\
\langle \boldsymbol{w}f,g\rangle =\langle f,\boldsymbol{w}^{\ast}g\rangle ,
\end{gather*}
and this expression follows from $\boldsymbol{w}=t^{N-1}\boldsymbol{T}_{N-1}^{-1}\cdots \boldsymbol{T}_{2}^{-1}\boldsymbol{T}_{1}^{-1}\xi_{1}$. Next
we use hypothesis~\eqref{form3} to relate norms for polynomials of different bosonic degrees.

\begin{Proposition}
Suppose $E\in\mathcal{Y}_{0}\cup\mathcal{Y}_{1}$, $\alpha\in\mathbb{N}_{0}^{N}$ then
\[
\Vert M_{\Phi\alpha,E}\Vert^{2}=\frac{1-q^{\alpha_{1}+1}t^{c(r_{\alpha}(1),E)}}{1-q}\Vert M_{\alpha,E}\Vert^{2}.
\]
\end{Proposition}

\begin{proof}
In~\eqref{form3} set $g=M_{\alpha,E}$ and $f=M_{\Phi\alpha,E}$ then $(1-q) \langle f,x_{N}\boldsymbol{w}g\rangle =(1-q) \Vert M_{\Phi\alpha,E}\Vert^{2}$. On~the other hand%
\begin{gather*}
\boldsymbol{D}_{N}f =\frac{1}{x_{N}}(1-\xi_{N}) f=\frac{1}{x_{N}}(1-\zeta_{\Phi\alpha,E}(N))
M_{\Phi\alpha,E} =(1-\zeta_{\Phi\alpha,E}(N)) \boldsymbol{w}M_{\alpha,E},
\\
\langle \boldsymbol{w}^{-1}\boldsymbol{D}_{N}f,g\rangle
=(1-\zeta_{\Phi\alpha,E}(N))\langle M_{\alpha,E},M_{\alpha,E}\rangle,
\end{gather*}
thus $\Vert M_{\Phi\alpha,E}\Vert^{2}=\frac{1-\zeta_{\Phi\alpha,E}(N)}{1-q}\Vert M_{\alpha,E}\Vert^{2}$ and~$\zeta_{\Phi\alpha,E}(N) =q\zeta_{\alpha,E}(1)
=q^{\alpha_{1}+1}t^{c(r_{\alpha}(1),E)}$.
\end{proof}

With this formula and Proposition~\ref{Mnorm1} we can use induction to find
$\Vert M_{\alpha,E}\Vert^{2}$ for any $\alpha$. The~first step
uses $\alpha=\boldsymbol{0}$ and any $E\in\mathcal{Y}_{0}\cup\mathcal{Y}_{1}$,
where $M_{\boldsymbol{0},E}=\tau_{E}$ and the spectral vector $[
\zeta_{\boldsymbol{0},E}(i) ]_{i=1}^{N}=\big[t^{c(i,E)}\big]_{i=1}^{N}$. Then $\Phi\alpha=(0,\dots,0,1)$ and $M_{\Phi\alpha,E}=x_{N}T^{(N)}\tau_{E}$, $\Vert M_{\Phi\alpha,E}\Vert^{2}=\frac{1-qt^{c(1,E)}}{1-q}\Vert\tau_{E}\Vert^{2}$ (see Propositions~\ref{tau0norm} and~\ref{tauEnorm} for this value).

The argument for establishing the formula for $\Vert M_{\lambda,E}\Vert^{2}$, where $\lambda\in\mathbb{N}_{0}^{N,+}$ uses the following
steps, starting with the assumption that $\lambda_{k}\geq1$ and $\lambda_{j}=0$ for $k<j\leq N$. Throughout $E$ is fixed. Let
\begin{gather*}
\mu =( \lambda_{1},\dots,\lambda_{k-1},\lambda_{k}-1,0..,0)
=\alpha^{+},
\\
\alpha =( \lambda_{k}-1,\lambda_{1},\dots,\lambda_{k-1}%
,0,\dots,0) ,
\\
\beta =( \lambda_{1},\dots,\lambda_{k-1},0,\dots,0,\lambda
_{k}) =\Phi\alpha
\end{gather*}
so that $\Vert M_{\alpha,E}\Vert^{2}=\mathcal{R}(\alpha,E)^{-1}\Vert M_{\mu,E}\Vert^{2}$, $\Vert
M_{\beta,E}\Vert^{2}=\frac{1-q^{\lambda_{k}}t^{c(k,E)}}{1-q}\Vert M_{\alpha,E}\Vert^{2}$ (since $r_{\alpha}(1) =k$) and $\Vert M_{\lambda,E}\Vert^{2}=\mathcal{R}(\beta,E) \Vert M_{\beta,E}\Vert^{2}$. In~the resulting formula we use a slightly different expression for~$u(z) =\frac{(t-z) (1-tz)}{(z-1)^{2}}=t\frac{( 1-z/t)(1-tz)}{(z-1)^{2}}$ because $u(z)/t$ is invariant under
$t\rightarrow\frac{1}{t}$, $z\rightarrow\frac{1}{z}$, but this causes a power of
$t$ to appear in the form $k(\lambda) :=\sum_{i=1}^{N}(N-2i+1) \lambda_{i}$ for $\lambda\in\mathbb{N}_{0}^{N,+}$. The shifted
$q$-factorial $( z;q)_{0}=1,( z;q)_{n+1}=(z;q)_{n}( 1-zq^{n})$, $n=0,1,2,\dots$ is used.

\begin{Theorem}[\cite{Dunkl2019}]
Suppose $\lambda\in\mathbb{N}_{0}^{N,+}$, $\alpha,\beta
\in\mathbb{N}_{0}^{N}$ and $E,F\in\mathcal{Y}_{0}\cup\mathcal{Y}_{1}$ then the
following satisfy the hypotheses~\eqref{form1}, \eqref{form2}, and~\eqref{form3}$:$%
\begin{gather*}
\langle M_{\alpha,E},M_{\beta,F}\rangle =0,\qquad (\alpha,E) \neq(\beta,F),
\\[.5ex]
\Vert M_{\alpha,E}\Vert^{2} =\mathcal{R}(\alpha,E)^{-1}\Vert M_{\alpha^{+},E}\Vert^{2},
\\
\Vert M_{\lambda,E}\Vert^{2} =t^{k(\lambda)}\Vert\tau_{E}\Vert^{2}(1-q)^{-\vert\lambda\vert} \prod\limits_{i=1}^{N}\big(qt^{c(i,E)};q\big)_{\lambda_{i}}
\\ \hphantom{\Vert M_{\lambda,E}\Vert^{2} =}
{} \times\prod\limits_{1\leq i<j\leq N}\frac{\big(qt^{c(i,E)-c(j,E) -1};q\big)_{\lambda_{i}-\lambda_{j}} \big(qt^{c(i,E) -c(j,E) +1};q\big)_{\lambda_{i}-\lambda_{j}}}
{\big(qt^{c(i,E) -c(j,E)};q\big)_{\lambda_{i}-\lambda_{j}}^{2}}.
\end{gather*}
\end{Theorem}

As Griffeth~\cite{Griffeth2010} pointed out there is not much cancellation between
successive terms in~gene\-ral; there is a certain amount for the extreme cases
$E_{0}=\{N-m,N-m+1,\dots,N\}$ and $E_{1}=\{1,2,\dots,m-1\}$. By~\cite[Proposition~11]{Dunkl2019} $\Vert M_{\alpha,E}\Vert^{2}>0$ if $q>0$ and $\min\big(q^{-1/N},q^{1/N}\big)
<t<\max\big(q^{-1/N},q^{1/N}\big)$.

\section{Symmetric Macdonald superpolynomials}\label{SyMP}

\subsection{From nonsymmetric to symmetric}

This concerns polynomials $p$ in $s\mathcal{P}_{m}$ which satisfy
$\boldsymbol{T}_{i}p=tp$ for $1\leq i<N$ and which are eigenfunctions of
$\sum_{i=1}^{N}\xi_{i}$. We call polynomials satisfying $\boldsymbol{T}%
_{i}p=tp$ for all $i$, or $\boldsymbol{T}_{i}p=-p$ for all $i$,
\textit{symmetric} or \textit{antisymmetric}, respectively. (The meaning of
\textit{symmetric} here is not the same as for the symmetric group situation,
as will be shown by example.) There are two approaches to~producing symmetric
polynomials. One way is to identify a set of $M_{\alpha,E}$ which is closed
under the steps $f\rightarrow(\boldsymbol{T}_{i}+b) f$ of the
type described in Proposition~\ref{Mtrans} and then to apply the symmetry
conditions to a linear combination of these polynomials with undetermined
coefficients. The~other way is to apply a symmetrization operator to one
polynomial. The~ori\-gi\-nal idea for these approaches comes from Baker and
Forrester~\cite{BakerForrester99}.

\begin{Definition}
For $\alpha\in\mathbb{N}_{0}^{N}$ and $E\in\mathcal{Y}_{0}\cup\mathcal{Y}_{1}$
let $\lfloor \alpha,E\rfloor $ denote the tableau obtained from
$Y_{E}$ by replacing $i$ by $\alpha_{i}^{+}$ for $1\leq i\leq N$. Let
$\mathcal{M}(\alpha,E) =\operatorname{span}\{M_{\beta,F}\colon \lfloor \beta,F\rfloor =\lfloor \alpha,E\rfloor\}$.
\end{Definition}

Example: let $N=9$, $m=4$, $E=\{2,3,6,8,9\}$, $\alpha=(3,5,6,2,2,1,4,4,6)$, $\alpha^{+}=(6,6,5,4,4,$ $3,2,2,1)$ and
\[
Y_{E}=\begin{bmatrix}
9 & 7 & 5 & 4 & 1\\
\cdot & 8 & 6 & 3 & 2
\end{bmatrix}\!,\qquad
\lfloor \alpha,E\rfloor =\begin{bmatrix}
1 & 2 & 4 & 4 & 6\\
\cdot & 2 & 3 & 5 & 6
\end{bmatrix}\!.
\]

\begin{Theorem}[{\cite[Proposition~5.2]{DunklLuque2012}}]
Suppose $\alpha\in\mathbb{N}_{0}^{N}$ and
$E\in\mathcal{Y}_{0}\cup\mathcal{Y}_{1}$, then there is a series of
transformations of the form $a(\boldsymbol{T}_{i}+b) $ mapping
$M_{\alpha,E}$ to $M_{\beta,F}$ if and only if $\lfloor \beta,F\rfloor =\lfloor \alpha,E\rfloor$.
\end{Theorem}

It is a consequence of the transformation rules that if $\lfloor \beta,F\rfloor =\lfloor \alpha,E\rfloor $ then the spectral
vector~$\zeta_{\beta,F}$ is a permutation of $\zeta_{\alpha,E}$. Furthermore
$\mathcal{M}(\alpha,E) $ is an $\mathcal{H}_{N}(t)$-module.

\begin{Theorem}[{\cite[Theorem~5.27]{DunklLuque2012}}]
Suppose $\alpha\in\mathbb{N}_{0}^{N}$ and $E\in\mathcal{Y}_{0}\cup\mathcal{Y}_{1}$ and $\lfloor\alpha,E\rfloor$ is column-strict $($the entries in column 1 are strictly
decreasing$)$ then there is a unique symmetric polynomial $($up~to multiplication
by a constant$)$ in $\mathcal{M}(\alpha,E) $ otherwise there is no
nonzero symmetric polynomial.
\end{Theorem}

In~\cite{Derosiersetal2003} the authors defined a \textit{superpartition} with $N$ parts
and fermionic degree $m$ as an $N$-tuple $\big(\Lambda_{1},\dots,\Lambda_{m};\Lambda_{m+1},\dots,\Lambda_{N}\big)$ which satisfies
$\Lambda_{1}>\Lambda_{2}>\cdots>\Lambda_{m}$ and $\Lambda_{m+1}\geq
\Lambda_{m+2}\geq\cdots\geq\Lambda_{N}$. Suppose $\lambda\in\mathbb{N}_{0}^{N,+}$, $E\in\mathcal{Y}_{0}$ and $\lfloor \lambda,E\rfloor$
is column strict, then $\Lambda_{i}=\lfloor \lambda,E\rfloor[m+2-i,1]$ for~$1\leq i$ $\leq m$ and $\Lambda_{i}=\lfloor\lambda,E\rfloor [ 1,N+1-i] $ for~$m+1\leq i\leq N$, and
also $\Lambda_{m}>\Lambda_{N}$. Alternatively suppose $\lambda\in
\mathbb{N}_{0}^{N,+}$, $E\in\mathcal{Y}_{1}$ and $\lfloor \lambda
,E\rfloor $ is column strict, then $\Lambda_{i}=\lfloor
\lambda,E\rfloor [m+1-i,1] $ for~$1\leq i\leq m$ and~$\Lambda_{i}=\lfloor \lambda,E\rfloor [1,N+2-i] $ for~$m+1\leq i\leq N$, and also $\Lambda_{m}\leq\Lambda_{N}$ (because $\Lambda_{m}=\lfloor \lambda,E\rfloor[1,1]$ and~$\Lambda_{N}=\lfloor \lambda,E\rfloor [1,2]$). Thus
the inequalities $\Lambda_{m}>\Lambda_{N}$ and $\Lambda_{m}\leq\Lambda_{N}$
distinguish $\mathcal{Y}_{0}$ from $\mathcal{Y}_{1}$.

As a standardization for the labels use $\lambda=\alpha^{+}$ and for $E$ use
the \textit{root} $E_{R}$ or the sink $E_{S}$
\begin{Definition}
Suppose $E\in\mathcal{Y}_{0},\lambda\in\mathbb{N}_{0}^{N,+}$ then the root
$E_{R}$ and the sink $E_{S}$ (which implicitly depend on $\lambda$) satisfy%
\begin{gather*}
\operatorname{inv}(E_{R}) =\min\big\{\operatorname{inv}(F) \colon \lfloor \alpha,F\rfloor
=\lfloor \lambda,E\rfloor \big\},
\\
\operatorname{inv}( E_{S}) =\max\big\{\operatorname{inv}(F) \colon \lfloor \alpha,F\rfloor
=\lfloor \lambda,E\rfloor \big\}.
\end{gather*}
The root and the sink are produced by minimizing the entries of $F$ in row 1,
respectively minimizing the entries of $F$ in column 1. For $E\in
\mathcal{Y}_{1}$ the definitions of $E_{R}$ and $E_{S}$ are reversed.
\end{Definition}

So in the above example $E_{R}=E$ and $E_{S}=\{1,3,6,7,9\} $ and
there are four sets $F$ such that $\lfloor \lambda,F\rfloor=\lfloor \lambda,E_{S}\rfloor$.

Consider $p=\sum_{\lfloor \beta,F\rfloor =\lfloor\lambda,E_{R}\rfloor}A(\beta,F) M_{\beta,F}$ then the
action of $\boldsymbol{T}_{i}$ decomposes the sum into pairs and singletons.
Suppose $\lfloor \beta,F\rfloor =\big\lfloor \lambda,E_{R}\big\rfloor$ for some $\beta$ with $\beta_{i}<\beta_{i+1}$, some $i$. Let $z=\zeta_{\beta,F}(i+1) /\zeta_{\beta,F}(i)$ then%
\begin{gather*}
\bigg(\boldsymbol{T}_{i}+\frac{t-1}{z-1}\bigg) M_{\beta,F} =M_{s_{i}\beta,F},\qquad
\boldsymbol{T}_{i}M_{\beta,F}=-\frac{t-1}{z-1}M_{\beta,F}+M_{s_{i}\beta,F},
\\
\bigg(\boldsymbol{T}_{i}+\frac{t-1}{z^{-1}-1}\bigg) M_{s_{i}\beta,F}=u(z) M_{\beta,F},\qquad
\boldsymbol{T}_{i}M_{s_{i}\beta,F}=u(z) M_{\beta,F}-\frac{t-1}{z^{-1}-1}M_{s_{i}\beta,F},
\end{gather*}
and
\[
(\boldsymbol{T}_{i}-t) \big(A(\beta,F)M_{\beta,F}+A(s_{i}\beta,F) M_{s_{i}\beta,F}\big) =0
\]
implies $A(\beta,F) =\frac{t-z}{1-z}A(s_{i}\beta,F)$.

\begin{Definition}
Let $u_{0}(z) :=\frac{t-z}{1-z}$, $u_{1}(z):=\frac{1-tz}{1-z}$ and for $\beta\in\mathbb{N}_{0}^{N}$, $F\in\mathcal{Y}_{0}\cup\mathcal{Y}_{1}$, $k=1,2$ let
\[
\mathcal{R}_{k}(\beta,F) =\prod\limits_{1\leq i<j\leq N,\,\beta_{i}<\beta_{j}}u_{k}
\big(q^{\beta_{j}-\beta_{i}}t^{c(r_{\beta}(j),F) -c(r_{\beta}(i),F)}\big).
\]
\end{Definition}

Thus $\mathcal{R}(\beta,F) =\mathcal{R}_{0}(\beta,F) \mathcal{R}_{1}(\beta,F)$ (see Definition~\ref{Rdef}).

\begin{Lemma}
If $\boldsymbol{T}_{i}p=tp$ for $1\leq i<N$ then
\[
A(\beta,F) =\mathcal{R}_{0}(\beta,F) A\big(\beta^{+},F\big).
\]
\end{Lemma}

\begin{proof}
Suppose $\beta_{i}>\beta_{i+1}$ then%
\[
\frac{\mathcal{R}_{0}(s_{i}\beta,F)}{\mathcal{R}_{0}(\beta,F)}
=u_{0}\bigg(\frac{\zeta_{\beta,F}(i)}{\zeta_{\beta,F}(i+1)}\bigg).
\]
The same argument as in Proposition~\ref{tau0norm} applies.
\end{proof}

\begin{Definition}
For $k=0,1$ let $\mathcal{C}_{k}(E) :=\prod\limits_{\substack{1\leq i<j<N,\\c(i,E) <0<c(j,E)}} u_{k}\big(t^{v(i,E) -c(ij,E)}\big)$.
\end{Definition}

Thus $\mathcal{C}\big[c(i,E)]_{i=1}^{N}\big) =\mathcal{C}_{0}(E) \mathcal{C}_{1}(E)$ (see~(\ref{Cprod1})).

\begin{Lemma}
$\dfrac{A(\beta,F)}{\mathcal{C}_{0}(F)}%
=\dfrac{A\left(\beta,E_{S}\right)}{\mathcal{C}_{0}\left(E_{S}\right)}$.
\end{Lemma}

\begin{proof}
Consider the possibilities when $\beta_{i}=\beta_{i+1}$ and $j=r_{\beta}(i)$: if $c(j,F) =c( j+1,F)+1$,
that is, $j$ and $j+1$ are in adjacent cells of row 1 of $Y_{F}$ then
$\boldsymbol{T}_{i}M_{\beta,F}=tM_{\beta,F}$, imposing no conditions on
$A(\beta,F) $; if $c(j,F) =c( j+1,F)-1$ then $\boldsymbol{T}_{i}M_{\beta,F}=-M_{\beta,F}$ but this occurs only if there are adjacent equal values ($\beta_{i}$) in column 1 of $\lfloor
\lambda,E_{R}\rfloor $, ruled out by hypothesis; $c(j,F)
<0<c( j+1,F) $. In~this case we relate $M_{\beta,F}$ to
$M_{\beta,s_{j}F}$, where $\operatorname{inv}( s_{j}F) =\operatorname{inv}%
(F) -1$: the formulas similar to~\eqref{TiM2} with
$z=\frac{\zeta_{\beta,F}(i) }{\zeta_{\beta,F}(i+1)
}=t^{c(j,F) -c( j+1,F) }$ appear here:%
\begin{gather*}
\boldsymbol{T}_{i}M_{\beta,s_{j}F} =-\frac{t-1}{z-1}M_{\beta,s_{j}F}+M_{\beta,F},
\\
\boldsymbol{T}_{i}M_{\beta,F} =u(z) M_{\beta,s_{j}F}-\frac{t-1}{z^{-1}-1}M_{\beta,F},
\end{gather*}
then $(\boldsymbol{T}_{i}-t) \big(A(\beta,F)M_{\beta,F}+A(\beta,s_{j}F) M_{\beta,s_{j}F}\big) =0$ implies $A\big(\beta,s_{j}F\big) \allowbreak=\frac{t-z}{1-z}A(\beta,F)$.
\end{proof}

So $A(\beta,F) =\mathcal{R}_{0}(\beta,F) A(\beta^{+},F) =A(\lambda,E_{S})\dfrac{\mathcal{R}_{0}(\beta,F) \mathcal{C}_{0}(E_{S})}{\mathcal{C}_{0}(F) }$.

\begin{Theorem}
Suppose $\lambda\in\mathbb{N}_{0}^{N,+}$, $E\in\mathcal{Y}_{0}$, and
$\lfloor \lambda,E\rfloor $ is column-strict then
\[
p_{\lambda,E}=\sum_{\left\lfloor \alpha,F\right\rfloor =\lfloor \lambda,E\rfloor }\frac{\mathcal{C}_{0}\left( E_{S}\right)
\mathcal{R}_{0}\left( \alpha,F\right) }{\mathcal{C}_{0}(F)
}M_{\alpha,F}%
\]
is the supersymmetric polynomial in $\mathcal{M}(\lambda,E) $,
unique when the coefficient of $M_{\lambda,E_{S}}$ is $1$.
\end{Theorem}

To show that this meaning of \textit{symmetric} is different from the group
case consider $N=4$, $E=\{1,2\}$, $m=3$, $\lambda=(2,1,0,0)$ and the corresponding symmetric polynomial (too large to display here) begins:
\[
p=x_{1}^{2}x_{2}\theta_{1}\theta_{2}(\theta_{3}+\theta_{4})-x_{1}^{2}x_{3}\theta_{1} (t\theta_{2}\theta_{3} +(t-1)\theta_{2}\theta_{4}-\theta_{3}\theta_{4}) -tx_{1}^{2}x_{4}\theta_{1}(\theta_{2}+\theta_{3}) \theta_{4}+\cdots\,.
\]

\subsection{Symmetrization operator and norms}

The symmetrization operator is defined analogously to the group case.

\begin{Definition}
For $n\geq1$ let $X_{0}=1$ and $X_{n}=1+\boldsymbol{T}_{n}X_{n-1}$, and
$S^{(n) }=X_{1}X_{2}\cdots X_{n}$.
\end{Definition}

Equivalently $X_{n}=1+\boldsymbol{T}_{n}+\boldsymbol{T}_{n}\boldsymbol{T}%
_{n-1}+\dots+\boldsymbol{T}_{n}\cdots\boldsymbol{T}_{2}\boldsymbol{T}_{1}$.

\begin{Theorem}
\label{symmop}
If $1\leq j\leq n$ then $\big(\boldsymbol{T}_{j}-t\big)S^{(n) }=0$.
\end{Theorem}

\begin{proof}
Consider the same formulas with $\boldsymbol{T}_{i}$ replaced by $s_{i}$ and
denote $\widetilde{X}_{n}=1+s_{n}\widetilde{X}_{n-1}$. In~the full expansion
there are $(n+1)!$ terms and the coefficient of $t^{k}$ in
$[n+1]_{t}!$ is the number of terms with $k$ factors. Claim
that $\widetilde{S}^{(n) }=\widetilde{X}_{1}\widetilde{X}_{2}\cdots\widetilde{X}_{n} =\sum_{u\in\mathcal{S}_{n+1}}u\in\mathbb{Z}\mathcal{S}_{n+1}$; proceeding by~induction the statement is true for $n=1$, where $\widetilde{X}_{1}=1+s_{1}$ and now suppose it is true for~$n$
and consider $\sum_{u\in\mathcal{S}_{n+1}}u(1+s_{n+1}+s_{n+1}s_{n}+\dots+s_{N+1}\cdots s_{1})$ acting on $\gamma=(\gamma_{1},\dots,\gamma_{n+2})$; then $s_{n+1}\cdots s_{j}
\gamma=(\gamma_{1},\dots,\gamma_{i-1},\gamma_{i+1},\dots,\gamma_{n+2},\gamma_{i})$. Thus $\sum_{u\in\mathcal{S}_{n+1}}us_{n+1}\cdots s_{j}$ is the sum of all~$u^{(i)}$ such that
$\big( u^{(i)}\gamma\big)_{n+2}=\gamma_{i}$. This shows
$\widetilde{S}^{(n+1) }=\sum_{u\in\mathcal{S}_{n+2}}u$.
Since the number of terms with~$k$ factors in $\widetilde{S}^{(n)}$ is the same as the number of $u$ of length $k$ each term is of~minimum length (the shortest expression of $u$ as a product of $\{s_{i}\}$). Thus replacing each $s_{i}$ by $\boldsymbol{T}_{i}$ shows
that~$S^{(n) }=\sum_{u\in\mathcal{S}_{n+1}}\boldsymbol{T}(u) $.

Replacing $\boldsymbol{T}_{i}$ by $\boldsymbol{T}_{n+1-i}$ for $1\leq i\leq n$
in $S^{(n) }$ does not affect the sum (implicitly the braid
relations are used). Given $j\leq n$ apply the map $\boldsymbol{T}%
_{i}\rightarrow\boldsymbol{T}_{j+1-i}$ in $X_{1}X_{2}\cdots X_{j}$ to obtain%
\[
S^{(n) }=\big( 1+\boldsymbol{T}_{j}\big) \big(
1+\boldsymbol{T}_{j-1}+\boldsymbol{T}_{j-1}\boldsymbol{T}_{j}\big)
\cdots\big( 1+\boldsymbol{T}_{1}+\dots+\boldsymbol{T}_{1}\cdots
\boldsymbol{T}_{j}\big) X_{j+1}\cdots X_{n}%
\]
and it is now obvious that $\big( \boldsymbol{T}_{j}-t\big) S^{(n) }=0$.
\end{proof}

\begin{Corollary}
Suppose $f\in s\mathcal{P}_{m}$ then $\boldsymbol{T}_{j}\big( S^{(N-1) }f\big) =tS^{(N-1) }f$ for $1\leq j<N$.
\end{Corollary}

\begin{Corollary}\label{S2S}
$S^{(n) }S^{(n) }=[n+1]_{t}!S^{(n)}$.
\end{Corollary}

\begin{proof}
The effect of $X_{j}$ on an invariant polynomial is to multiply by
$1+t+t^{2}+\dots+t^{j}$.
\end{proof}

\begin{Corollary}
Suppose $f,g\in s\mathcal{P}_{m}$ then $\big\langle S^{(N-1)}f,g\big\rangle =\big\langle f,S^{(N-1)}g\big\rangle$.
\end{Corollary}

\begin{proof}
Suppose $u\in\mathcal{S}_{N}$ and $u=s_{i_{1}}\cdots s_{i_{\ell}}$ is a
shortest expression for $u$ so that $\boldsymbol{T}(u)
=\boldsymbol{T}_{i_{1}}\cdots$ ${}\times\boldsymbol{T}_{i_{\ell}}$ then $\langle
\boldsymbol{T}(u) f,g\rangle =\langle
f,\boldsymbol{T}_{i_{\ell}}\cdots\boldsymbol{T}_{i_{1}}g\rangle
=\big\langle f,\boldsymbol{T}(u^{-1}) g\big\rangle $. Since
$\sum_{u\in\mathcal{S}_{N}}\!\boldsymbol{T}(u)
=\sum_{u\in\mathcal{S}_{N}}\!\boldsymbol{T}(u^{-1}) $ this
completes the proof.
\end{proof}

There is a summation-free formula for $\Vert p_{\lambda,E}\Vert^{2}$, derived as follows:

Suppose $\lfloor \alpha,F\rfloor =\lfloor \lambda,E\rfloor$ then $S^{(N-1)}M_{\alpha,F}=cp_{\lambda,E}$
for some constant $c$, because of the uniqueness of $p_{\lambda,E}$ in
$\mathcal{M}(\lambda,E)$. Then
\begin{gather}
\big\langle p_{\lambda,E},S^{(N-1) }M_{\alpha,F}\big\rangle
 =c\langle p_{\lambda,E},p_{\lambda,E}\rangle
 =\big\langle S^{(N-1) }p_{\lambda,E},M_{\alpha,F}\big\rangle \nonumber
\\ \hphantom{\big\langle p_{\lambda,E},S^{(N-1) }M_{\alpha,F}\big\rangle}
{} =[N]_{t}!\langle p_{\lambda,E},M_{\alpha,F}\rangle =[N]_{t}!\frac{\mathcal{C}_{0} \big(E_{S}\big)\mathcal{R}_{0}(\alpha,F)}{\mathcal{C}_{0}(F)}\Vert M_{\alpha,F}\Vert^{2}.
 \label{pSymM}
\end{gather}
The evaluation depends on determining $c$, which can be done by using
$M_{\lambda^{-},E_{R}}$, where $\lambda^{-}$ is the nondecreasing rearrangement
of $\lambda$. For each $i\leq\lambda_{1}$ let $m_{i}=\#\big\{j\colon \big\lfloor
\lambda,E_{S}\big\rfloor [1,j] =i\big\}$ (the~mul\-ti\-plicity
of $i$ in row 1 of $\big\lfloor \lambda,E_{S}\big\rfloor$). We will show
that the coefficient of $M_{\lambda,E_{S}}$ in $S^{(N-1)
}M_{\lambda_{-},E_{R}}$ is~$[m_{i}]_{t}!$. (This was shown in
\cite[Theorem~5.39]{DunklLuque2012}; we are outlining a proof here with simplifications
due to the simple hook shape $(N-m,1^{m})$, also to accommodate
the different notation.) Here is an illustration of the following theorem and
the method of proof. Suppose $\lambda=(3,2,2,2,1,0)$ and
\[
Y_{E_{R}}=\begin{bmatrix}
6 & 5 & 3 & 2\\
\cdot & 4 & 1 &
\end{bmatrix}\!,\qquad
\big\lfloor \lambda,E_{S}\big\rfloor =\begin{bmatrix}
0 & 1 & 2 & 2\\
\cdot & 2 & 3 &
\end{bmatrix}\!,\qquad
Y_{E_{S}}=\begin{bmatrix}
6 & 5 & 4 & 3\\
\cdot & 2 & 1 &
\end{bmatrix}\!.
\]
We demonstrate the effect on the significant terms by using the spectral vectors:
\begin{gather*}
\big(\lambda^{-},E_{R}\big) \simeq\big(1,qt,q^{2}t^{3},q^{2}t^{2},q^{2}t^{-1},q^{3}t^{-2}\big),
\\
X_{5} \colon\ \big( qt,q^{2}t^{3},q^{2}t^{2},q^{2}t^{-1},q^{3}t^{-2},1\big),
\\
X_{4} \colon\ \big( q^{2}t^{3},q^{2}t^{2},q^{2}t^{-1},q^{3}t^{-2},qt,1\big),
\\
X_{3} \colon\ (1+t) \times\big( q^{2}t^{3},q^{2}t^{-1},q^{3}t^{-2},q^{2}t^{2},qt,1\big),
\\
X_{2} \colon\ (1+t) \times\big( q^{2}t^{-1},q^{3}t^{-2},q^{2}t^{3},q^{2}t^{2},qt,1\big),
\\
X_{1} \colon\ (1+t) \times\big( q^{3}t^{-2},q^{2}t^{-1},q^{2}t^{3},q^{2}t^{2},qt,1\big),
\end{gather*}
and this is the spectral vector of $\big(\lambda,E_{S}\big) $.

\begin{Theorem}\label{SymM}
For each $M_{\beta,F}\in\mathcal{M}\big(\lambda,E_{S}\big) $
with $(\beta,F) \neq\big(\lambda,E_{S}\big)$ there is a
constant $c_{\beta,F}$ such~that
\[
S^{(N-1) }M_{\lambda^{-},E_{R}}=\prod\limits_{i=1}^{\lambda_{1}}[m_{i}]_{t}!M_{\lambda,E_{S}} +\sum_{\lfloor \beta,F\rfloor =\left\lfloor \lambda,E_{S}\right\rfloor}
\big\{c_{\beta,F}M_{\beta,F}\colon (\beta,F) \neq\big(\lambda,E_{S}\big)\big\}.
\]
\end{Theorem}

\begin{proof}
The proof relies on identifying the intermediate steps in transforming
$x^{\lambda^{-}}R_{\lambda_{-}}(\tau_{E_{R}})$ to~$x^{\lambda
}\tau_{E_{S}}$. Roughly the action of $X_{N-1-i}$ transforms $M_{\alpha(i),E}$ to $M_{\alpha(i+1) ,E}$ by means of
$\boldsymbol{T}_{N-1-i}$ ${}\times\cdots\boldsymbol{T}_{1}$, where $\alpha(0) =\lambda^{-}$ and for $i\geq1$%
\[
\alpha(i) =\big( \lambda_{N-i},\lambda_{N-2},\dots,\lambda_{1},\lambda_{N-i+1},\dots,\lambda_{N}\big)
\]
however the situation is not this simple because repeated values of
$\lambda_{j}$ have to taken into account. Note that $X_{j}$ does not affect
the variables $x_{k}$ for $k>j+1$. It (almost) suffices to consider the
coefficient of $x^{\alpha(1) }$ in $X_{N-1}M_{\lambda_{-},E_{R}}$. (Throughout we use $\Sigma$ to denote a linear combination of~terms
$M_{\beta,F}$ which can not be transformed into $M_{\lambda,E_{S}}$ by the
operators $X_{1}\cdots X_{j}$.) Suppose that $\lambda_{N}<\lambda_{N-1}$ (that
is, $\lambda_{1}^{-}<\lambda_{2}^{-}$) then $\boldsymbol{T}_{N-1-i}%
\cdots\boldsymbol{T}_{1}M_{\lambda^{-},E_{R}}=M_{\alpha(1)
,E_{R}}+\Sigma$, and the process is repeated with $M_{\alpha(1)
,E_{R}}.$ The other possibility is that $\lambda_{N-k}>\lambda_{N-k+1}%
=\cdots=\lambda_{N}$ for some $k\geq2$. This implies $r_{\lambda^{-}}(i) =N-k+i$ for $1\leq i\leq k$ and $c\left( r_{\lambda^{-}}(i) ,E_{R}\right) =k-i$, because the entries $N-k+1,\dots,N$ are
adjacent in row 1 of $Y_{E_{R}}$ (by hypothesis $\lfloor \lambda,E\rfloor [2,1] >\lambda_{N}$ and so~$\lambda_{N}=\lambda_{N-1}$ implies $\lfloor \lambda,E\rfloor [1,2] =N-1$). Thus $\boldsymbol{T}_{i}M_{\lambda^{-},E_{R}}=tM_{\lambda^{-},E_{R}}$ for $1\leq i\leq k-1$ and
\begin{gather*}
\boldsymbol{T}_{N-1}\cdots\boldsymbol{T}_{k}( 1+\boldsymbol{T}_{k-1} +\boldsymbol{T}_{k-1}\boldsymbol{T}_{k-2}+\dots+\boldsymbol{T}_{k-1}\cdots\boldsymbol{T}_{1}) M_{\lambda^{-},E_{R}}
\\ \qquad
{} =\big(1+t+t^{2}+\dots+t^{k-1}\big) \boldsymbol{T}_{N-1}\cdots\boldsymbol{T}_{k}M_{\lambda^{-},E_{R}}
=[k]_{t}\big\{ M_{\alpha(1) ,E_{R}}+\Sigma\big\} .
\end{gather*}
Then $\alpha(1) =\big(\lambda_{N-1},\dots,\lambda
_{N-k+1},\lambda_{N-k},\dots,\lambda_{1},\lambda_{N}\big) $ and the
previous argument applies with $k-1$ replacing $k$. The result of applying
$X_{N-k}\cdots X_{N-1}$ is $[k]_{t}M_{\alpha(k)
,E^{r}}+\Sigma$. Now $\alpha(k) =\big( \lambda_{N-k}%
,\lambda_{N-k+1},\cdots\big) $ and the $\lambda_{N}<\lambda_{N-1}$ type
process applies.

The last case to consider is $\lambda_{N-i-1-k}>\lambda_{N-i-k}=\dots
=\lambda_{N-i}>\lambda_{N-i+1}$, where $Y_{E_{R}}[\ell,1] =N-i$
and the entries $N-i-1,N-i-2,\dots,N-i-k$ are adjacent in row 1 of $Y_{E_{R}%
}$. Then
\begin{gather*}
\alpha(i) =\big(\lambda_{N-i},\dots,\lambda_{N-i-k},\lambda_{N-i-k-1}\big),
\end{gather*}
and similarly to the previous case $\boldsymbol{T}_{i}M_{\alpha(i) ,E^{\prime}}=tM_{\alpha(i) ,E^{\prime}}$ for $1\leq i\leq k-1$. The set $E^{\prime}$ is an~inter\-mediate step in a series of
transpositions transforming $E_{R}$ to $E_{S}$, at this stage using only~$s_{j}$ with $i>N-i$. Similarly
\begin{gather*}
 \boldsymbol{T}_{N-i-1}\cdots\boldsymbol{T}_{k}(1+\boldsymbol{T}_{k-1} +\boldsymbol{T}_{k-1}\boldsymbol{T}_{k-2}+\dots+\boldsymbol{T}_{k-1}\cdots\boldsymbol{T}_{1}) M_{\alpha(i) ,E^{\prime}}
 \\ \qquad
 =\big( 1+t+t^{2}+\dots+t^{k-1}\big) \boldsymbol{T}_{N-i-1}
\cdots\boldsymbol{T}_{k}M_{\alpha(i) ,E^{\prime}} =[k]_{t}\big\{ M_{\alpha(i+1),E"}+\Sigma\big\} .
\end{gather*}
Here $\boldsymbol{T}_{k}$ transforms $M_{\alpha(i) ,E^{\prime}}$
to $M_{\alpha(i) ,E''}$, where $E''=s_{N-i-1}E^{\prime}$ (since
$c(N-i,E^{\prime}) <0<c(N-i-1,E^{\prime}) $ and
$\operatorname{inv}(E'') =\operatorname{inv}(E^{\prime}) +1$.
Eventually these steps transform $E_{R}$ to~$E_{S}$ and~$\lambda^{-}$ to~$\lambda$. Each set of $m_{i}$ (contiguous) $\lambda_{i}$ values $\lambda
_{j}=i$ in row 1 of $\big\lfloor \lambda,E_{S}\big\rfloor $ contributes a
factor of~$[m_{i}]_{t}!$. By beginning with $E_{R}$ the factors
appearing in $(\boldsymbol{T}_{i}+b) M_{\beta,F}$ are always $1$
(see~(\ref{TiM1}) and~(\ref{TiM2})).
\end{proof}

\begin{Lemma}
Suppose\vspace{-.5ex}
\[
F(\alpha,E) :=\prod\limits_{\substack{1\leq i<j\leq
N\\\alpha_{i}<\alpha_{j},}}g\big(\alpha_{j}-\alpha_{i},c(r_{\alpha}(j),E) -c(r_{\alpha}(i,E))\big)
\]
for some function $g$ then\vspace{-.5ex}
\[
F(\lambda^{-},E) =\prod\limits_{\lambda_{i}>\lambda_{j}}g
\big(\lambda_{i}-\lambda_{j},c(i,E) -c(j,E) \big).
\]
\end{Lemma}

\begin{proof}
It is clear that $\lambda_{i}>\lambda_{j}$ if and only if $\lambda_{N+1-i}%
^{-}>\lambda_{N+1-j}^{-}$. If $\lambda_{a-1}>\lambda_{a}=\cdots=\lambda
_{a+k-1}>\lambda_{a+k}$ then $[r_{\lambda^{-}}(b+i)]_{i=1}^{k}=[a,a+1,\dots,a+k-1]$ and $\lambda
_{b+i}^{-}=\lambda_{a}$ for $b=N+1-a-k,1\leq i\leq k$. The corresponding
contents $[c(r_{\lambda^{-}}(b+i)),E]_{i=1}^{k}$ are the same as $[c(a+i-1),E]_{i=1}^{k}$. So each term in $F_{\lambda^{-},E}$ matches one in the
stated $\lambda$-product.
\end{proof}

\begin{Theorem}
Suppose $\lfloor \lambda,E_{S}\rfloor$ is column-strict and
$m_{i}=\#\big\{ j\colon \lfloor \lambda,E_{S}\rfloor [1,j] =i\big\}$ for $0\leq i$ $\leq\lambda_{1}$ then\vspace{-1ex}
\begin{gather*}
\Vert p_{\lambda,E_{S}}\Vert^{2} =t^{2(N-m-1)
+k(\lambda)}[m+1]_{t}(1-q)^{-\vert \lambda\vert }
\prod\limits_{i=1}^{N}\big( qt^{c\left(i,E_{S}\right)};q\big)_{\lambda_{i}}
\\ \hphantom{\Vert p_{\lambda,E_{S}}\Vert^{2} =}
{} \times\prod\limits_{1\leq i<j\leq N}\frac{\big(qt^{c\left(
i,E_{S}\right) -c\left( j,E_{S}\right) -1};q\big)_{\lambda_{i}%
-\lambda_{j}}\big( qt^{c\left( i,E_{S}\right) -c\left( j,E_{S}\right)
+1};q\big)_{\lambda_{i}-\lambda_{j}-1}}{\big(1-q^{\lambda_{i}%
-\lambda_{j}}t^{c\left( i,E_{S}\right) -c\left( j,E_{S}\right) }\big)
\big( qt^{c\left(i,E_{S}\right) -c\left(j,E_{S}\right) };q\big)
_{\lambda_{i}-\lambda_{j}-1}^{2}}
\\ \hphantom{\Vert p_{\lambda,E_{S}}\Vert^{2} =}
{} \times\frac{[N]_{t}!}{\prod\limits_{i\geq0}[m_{i}]_{t}!}\mathcal{C}_{0}\big(E_{S}\big) \mathcal{C}%
_{1}\big( E_{R}\big) .
\end{gather*}
\end{Theorem}

\begin{proof}
By (\ref{pSymM}) and Theorem~\ref{SymM} $c=\prod_{i\geq0}[m_{i}]_{t}!$,
\begin{align*}
\Vert p_{\lambda,E_{S}}\Vert^{2} &=\frac{[N]_{t}!}{\prod\limits_{i\geq0}[m_{i}]_{t}!}
\frac{\mathcal{C}_{0}\big( E_{S}\big) \mathcal{R}_{0}\big( \lambda^{-},E_{R}\big)
}{\mathcal{C}_{0}\big( E_{R}\big)}\big\Vert M_{\lambda^{-},E_{R}}\big\Vert^{2}
\\
 &=\frac{[N]_{t}!}{\prod\limits_{i\geq0}[m_{i}]_{t}!}\frac{\mathcal{C}_{0}\big(E_{S}\big) \mathcal{R}_{0}\big(\lambda^{-},E_{R}\big)}{\mathcal{C}_{0}\big(E_{R}\big)
\mathcal{E}\big(\lambda^{-},E_{R}\big)}\big\Vert M_{\lambda,E_{R}}\big\Vert^{2}%
\end{align*}
and
\begin{gather*}
\big\Vert M_{\lambda,E_{R}}\big\Vert^{2} =t^{\ell(\lambda)}\Vert \tau_{E_{R}}\Vert^{2}(1-q)
^{-\vert \lambda\vert }\prod\limits_{i=1}^{N}\big(qt^{c\left(i,E_{R}\right)};q\big)_{\lambda_{i}}
\\ \hphantom{\big\Vert M_{\lambda,E_{R}}\big\Vert^{2} =}
 {}\times\prod\limits_{1\leq i<j\leq N}\frac{\big(qt^{c\left(i,E_{R}\right) -c\left( j,E_{R}\right) -1};q\big)_{\lambda_{i}-\lambda_{j}}\big( qt^{c\left( i,E_{R}\right) -c\left( j,E_{R}\right)
+1};q\big)_{\lambda_{i}-\lambda_{j}}}{\big( qt^{c\left( i,E_{R}\right)
-c\left(j,E_{R}\right)};q\big)_{\lambda_{i}-\lambda_{j}}^{2}}.
\end{gather*}
Also $\Vert \tau_{E_{R}}\Vert^{2}=t^{2(N-m-1)}[m+1]_{t}\mathcal{C}\big([c(i,E_{R})]_{i=1}^{N}\big) =t^{2(N-m-1)}[m+1]_{t}\mathcal{C}_{0}\big(E_{R}\big) \mathcal{C}_{1}\big(E_{R}\big)$. Since $\mathcal{E}\big(\lambda^{-},E_{R}\big) =\mathcal{R}_{0}\big(\lambda^{-},E_{R}\big) \mathcal{R}_{1}\big(\lambda^{-},E_{R}\big)$ the terms $\mathcal{R}_{0}\big( \lambda^{-},E_{R}\big)$ and
$\mathcal{C}_{0}\big( E_{R}\big) $ cancel out. By the lemma
\begin{gather*}
\mathcal{R}_{1}\big(\lambda^{-},E_{S}\big) =\prod\limits_{1\leq
i<j\leq N}u_{1}\big( q^{\lambda_{i}-\lambda_{j}}t^{c\left( i,E_{R}\right)
-c\left( j,E_{R}\right) }\big)
 =\prod\limits_{1\leq i<j\leq N}\frac{1-q^{\lambda_{i}-\lambda_{j}%
}t^{c\left( i,E_{R}\right) -c\left( j,E_{R}\right) +1}}{1-q^{\lambda
_{i}-\lambda_{j}}t^{c\left( i,E_{R}\right) -c\left( j,E_{R}\right) }}%
\end{gather*}
and dividing by this product changes two of the $\left( \ast;q\right)
_{\lambda_{i}-\lambda j}$ terms to $\left( \ast;q\right)_{\lambda
_{i}-\lambda j-1}$.
\end{proof}

Note that $E_{R}$ can be replaced by $E_{S}$ in the first two lines of the
formula for $\Vert p_{\lambda,E_{S}}\Vert^{2}.$ By using the $M$
map the formulas produce supersymmetric polynomials in $\mathcal{P}_{m,1}$:
consider the polynomials $M(p_{\lambda,E_{S}})$, where
$p_{\lambda,E_{S}}\in\mathcal{P}_{m-1,0}$. This is why we do not go into
detail about the $E\in\mathcal{E}_{1}$ case. The norm formula implies the
identity%
\[
\sum_{\lfloor \alpha,F\rfloor =\lfloor \lambda,E\rfloor
}\frac{\mathcal{C}_{1}(F) \mathcal{R}_{0}(\alpha,F)}{\mathcal{C}_{0}(F) \mathcal{R}_{1} (\alpha,F)} =\frac{[N]_{t}!}{\prod\limits_{i\geq0}[m_{i}]_{t}!}\frac{\mathcal{C}_{1}\big( E_{R}\big)}
{\mathcal{C}_{0}\big( E_{S}\big) \mathcal{R}_{1}\big(\lambda^{-},E_{R}\big)}.
\]
This formula was checked by computer algebra for a ``small'' example,
$N=5$, $m=2$, $\lambda=(2,2,1,1,0)$ with
\[
\big\lfloor \lambda,E_{S}\big\rfloor =\begin{bmatrix}
0 & 1 & 2\\
\cdot & 1 & 2
\end{bmatrix}\!,\qquad
E_{R}=\{2,4,5\},\qquad
E_{S}=\{1,3,5\} ;
\]
there are $120$ labels $(\beta,F)$ with $\lfloor \beta,F\rfloor =\big\lfloor \lambda,E_{S}\big\rfloor$, that is $\dim\mathcal{M}\big(\lambda,E_{S}\big) =120$.

\subsection{Special values}

In~the scalar Macdonald polynomial situation there are formulas for special
values. There is one such fairly simple formula here. Let $F=\{1,2,\dots,m,N\}$. (In~the $\mathcal{Y}_{1}$-case use $E_{1}$.) Refer
to Proposition~\ref{Dnorm} for useful facts about $\tau_{F}$ and $\Vert\tau_{F}\Vert^{2}$.

\begin{Proposition}\label{pfact}
Suppose $p\in s\mathcal{P}_{m}$ is symmetric and%
\[
z:=\big( z_{1},z_{2},\dots,z_{m},t^{N-m-1},\dots,t^{2},t,1\big)
\]
then
\[
p(z;\theta) =\prod\limits_{1\leq i<j\leq m}(z_{i}-tz_{j})
\prod\limits_{k=1}^{m}\big(z_{k}-t^{N-m}\big)p_{0}(z_{1},\dots,z_{m})\tau_{F},
\]
where $p_{0}$ is $\mathcal{S}_{m}$-symmetric.
\end{Proposition}

\begin{proof}
First we show that if $\boldsymbol{T}_{i}p(x;\theta) =tp(x;\theta)$ then $T_{i}p\big( x^{(i)};\theta\big)=tp\big( x^{(i) };\theta\big)$, where $x_{i}^{(i) }=tx_{i+1}^{(i)}$. By hypothesis%
\begin{gather*}
(1+\boldsymbol{T}_{i}) p(x;\theta) =(1+t) p(x;\theta)
 =p(x;\theta) +(1-t) x_{i+1}\frac{p(x;\theta) -p(xs_{i};\theta) }{x_{i}-x_{i+1}}%
+T_{i}p( xs_{i};\theta).
\end{gather*}
Substitute $x_{i}=tx_{i+1}$ in the equations:
\begin{gather*}
(1+t) p(x;\theta) =p(x;\theta)-(p(x;\theta) -p(xs_{i};\theta))+T_{i}p( xs_{i};\theta)
 =( 1+T_{i}) p( xs_{i};\theta)
\end{gather*}
and this shows $(t-T_{i}) p(x;\theta) =0$ at
$x=x^{(i)}$. By hypothesis on $z$ this shows $T_{i}p(z;\theta) =tp(z;\theta) $ for $m+1\leq i<N$ and this implies $\omega_{i}p(z;\theta) =t^{N-i}p(z;\theta)$ for $m+1\leq i\leq N$. There is only one $\tau_{E}$ which has these
$\{\omega_{i}\} $-eigenvalues, namely $\tau_{F}$. Thus $p(z;\theta) =\widetilde{p}( z_{1},\dots,z_{m}) \tau_{F}$
for some polynomial $\widetilde{p}$ and $\boldsymbol{T}_{i}\widetilde
{p}(z) \tau_{F}=t\widetilde{p}(z) \tau_{F}$ for
$1\leq i<m$. In~this range $T_{i}\tau_{F}=-\tau_{F}$ thus $\widetilde{p}$
satisfies the equation%
\begin{gather*}
t\widetilde{p}(z) \tau_{E} =(1-t) z_{i+1}\frac{\widetilde{p}(z) -\widetilde{p}(zs_{i})
}{z_{i}-z_{i+1}}\tau_{E}-\widetilde{p}(zs_{i}) \tau_{E},
\\
\widetilde{p}(zs_{i}) =\frac{z_{i+1}-tz_{i}}{z_{i}-tz_{i+1}}\widetilde{p}(z) .
\end{gather*}
Thus $z_{i}-tz_{i+1}$ is a factor of $\widetilde{p}(z) $ because
$\widetilde{p}(zs_{i}) $ is polynomial. Furthermore
$\widetilde{p}(z) /(z_{i}-tz_{i+1})$ is~$s_{i}$-invariant. We claim by induction that $(z_{i}-tz_{i+k})$ is a~factor of $\widetilde{p}(z) $ for $1\leq i<i+k\leq m$: this is
valid for $k=1$ so consider that $(z_{i}-tz_{i+k})$ is a factor
of $\widetilde{p}(z) $ and $\widetilde{p}(z)/(z_{i+k}-tz_{i+k+1})$ is~$s_{i+k}$-invariant thus $(z_{i}-tz_{i+k+1})$ is a factor (where $i+k+1\leq m$).

Suppose $z_{m}^{\prime}=t^{N-m}=tz_{N-m+1}^{\prime}$ then $T_{m}\widetilde
{p}(z^{\prime})\tau_{F}=t\widetilde{p}(z^{\prime}) \tau_{F}$ but this implies $\widetilde{p}(z^{\prime})=0$ or~$\omega_{m}\tau_{F}=t^{N-m}\tau_{F}$ which is impossible. Thus $\big(z_{m}-t^{N-m}\big)$ is a factor of $\widetilde{p}(z) $. The
symmetry properties imply $\big(z^{i}-t^{N-m}\big)$ is a factor of
$\widetilde{p}(z) $ for $1\leq i\leq m$.
\end{proof}

Gonz\'{a}lez and Lapointe~\cite{GonzalezLapointe2020} proved an evaluation formula for the
version of supersymmetric Macdonald polynomials constructed in~\cite{BlondeauFournieretal2012}, with
\[
z=\big(t^{N-1}q^{-m},t^{N-2}q^{1-m},\dots,t^{N-m}q^{-1},t^{N-m-1},\dots,t,1\big).
\]
An example appears to show there is no such general result in our version.
However there may be one for the special case where $\lfloor
\lambda,E_{S}\rfloor [1,j] =0$ for $1\leq j\leq N-m$. At
this point we offer no conjecture, but some very small examples with $N=3$, $4$
and $\vert \lambda\vert \leq4$ suggest there is something to be found.

\subsection{Minimal symmetric polynomial}

For given $N,m$, isotype $(N-m,1^{m}) $ there is a unique
column-strict tableau with minimum sum of entries, namely
\[
\lfloor \lambda,E\rfloor =\begin{bmatrix}
0 & 0 & 0 & \cdots & 0\\
\cdot & 1 & 2 & \cdots & m
\end{bmatrix}\!,
\]
thus $\lambda=(m,m-1,\dots,2,1,0,\dots,0)$ and $E_{R}=E_{S}=\{1,2,\dots,m,N\}$. There is non-trivial multiplicity $m_{0}=N-m$. Thus $p_{\lambda,E}$ is the symmetric polynomial in
$\mathcal{P}_{m,0}$ of minimum bosonic degree $\frac{m(m+1)}{2}$ There is a concise formula for $\Vert p_{\lambda,E}\Vert^{2}$.

\begin{Theorem}\label{symnorm}
Suppose $\lambda=(m,m-1,\dots,1,0,\dots,0)
\in\mathbb{N}_{0}^{N,+}$ and $E=\{1,2,\dots,m,N\}$ then
$\lfloor \lambda,E\rfloor$ is column-strict and%
\[
\Vert p_{\lambda,E}\Vert^{2}=\frac{[N]_{t}![N]_{t} t^{\gamma}}{[N-m]_{t}![N-m]_{t}}\big(qt^{-N};q\big)_{m}\prod\limits_{j=2}^{m}\big(qt^{-j};q\big)_{j-1},
\]
with $\gamma=(N-m-1) \big(1+\frac{(m+1)(m+2)}{2}\big) +\frac{m(m+1) (m+2)}{6}$.
\end{Theorem}

\begin{proof}
The exponent on $t$ is%
\begin{align*}
\gamma & =2(N-m-1) +k(\lambda) +m(N-m-1)
\\
& =(m+2) (N-m-1) +\sum_{i=1}^{m}(N-2i+1) (m+1-i).
\end{align*}
The spectral vector is $\zeta_{\lambda,E}=\big[q^{m}t^{-m},q^{m-1}%
t^{1-m},\dots,qt^{-1},t^{N-m-1},\dots,t,1\big]$. Consider the content
product (part of $\mathcal{C}$) $\Pi_{k}$ for the pairs $(m+1-k,m+j)$ with $1\leq k\leq m$ and $j=1,\dots,N-m-1$:
\begin{gather*}
\Pi_{k} =\!\!\!\prod\limits_{j=1}^{N-m-1}\!\!\!u\big(t^{-k-j}\big)
=\!\!\!\prod\limits_{j=1}^{N-m-1}\!\frac{1\!-\!t^{-k-j+1}}{1\!-t^{-k-j}}\frac{t\!-t^{-k-j}%
}{1\!-t^{-k-j}} =t^{N-m-1}\frac{1\!-\!t^{-k}}{1\!-t^{-k-N+m+1}}\frac{1\!-\!t^{-k-N+m}}{1\!-t^{-k-1}},
\end{gather*}
by telescoping the products (and $t-t^{-k-j}=t(1-t^{-k-j-1})$)
then%
\[
\prod\limits_{k=1}^{m}\Pi_{k}=t^{m(N-m-1)}\frac{1-t^{-1}}{1-t^{-m-1}} \frac{1-t^{-N}}{1-t^{-N+m}}=t^{m(N-m-1)}\frac{[N]_{t}}{[m+1]_{t}[N-m]_{t}}.
\]
Consider the $q$-factors for the pairs $(m+1-k,m+j)$ with
$1\leq k\leq m$ and $j=1,\dots,N-m-1$ (use telescoping)%
\begin{gather*}
P_{1,k} =\prod\limits_{j=1}^{N-m-1}\frac{(qt^{-k-j-1};q)_{k} \big(qt^{-k-j+1};q\big)_{k-1}}{\big(qt^{-k-j};q\big)_{k}\big(qt^{-k-j};q\big)_{k-1}}
 =\frac{\big( qt^{-k-N+m};q\big)_{k}\big( qt^{-k};q\big)_{k-1}} {(qt^{-m-1};q)_{m}\big(qt^{-k-N+m+1};q\big)_{k-1}},
\end{gather*}
then%
\[
P_{1}:=\prod\limits_{k=1}^{m}P_{1,k}=\frac{\big(qt^{-N};q\big)_{m}}{\big(qt^{-m-1};q\big)_{m}}.
\]
Next consider $(m+1-k,m+1-i)$ with $k>i\geq1$ together with $(m+1-k,N)$:%
\[
P_{2,k}=\prod\limits_{i=0}^{k-1}\frac{\big(qt^{-k+i-1};q\big)_{k-i}
\big(qt^{-k+i+1};q\big)_{k-i-1}}{\big( qt^{-k+i};q\big)_{k-i}
\big(qt^{-k+i};q\big)_{k-i-1}}=\frac{\big( qt^{-k-1};q\big)_{k}}{\big(qt^{-k};q\big)_{k}}.
\]
Combine%
\begin{align*}
\prod\limits_{i=1}^{N}\big(qt^{c(i,E)};q\big)_{\lambda_{i}}P_{1}\prod\limits_{k=1}^{m}P_{2,k}
& =\frac{\big(qt^{-N};q\big)_{m}}{\big( qt^{-m-1};q\big)_{m}}\prod\limits_{k=1}^{m} \frac{\big(qt^{-k};q\big)_{k}\big(qt^{-k-1};q\big)_{k}}{\big(qt^{-k};q\big)_{k}}
\\
&=\big(qt^{-N};q\big)_{m}\prod\limits_{k=1}^{m-1}\big( qt^{-k-1};q\big)_{k}.
\end{align*}
This concludes the proof.
\end{proof}

The formula has a hook length interpretation: $\prod_{(i,j) \in Y_{E}}\!\big( qt^{-\mathrm{hook}(i,j)};q\big)_{\mathrm{leg}(i,j) }$; here \mbox{$\mathrm{leg}(1,j)\! =\!0$} for $2\leq j\leq N-m$; $\mathrm{hook}(i,1)=m+2-i$, $\mathrm{leg}(i,1) =m+1-i$, $2\leq i\leq m+1$ and
$\mathrm{hook}(1,1) =N$, $\mathrm{leg}(1,1) =m$ (see~\cite[Theorem~6.22]{DunklLuque2012}).

\subsection{\label{antiSy}Antisymmetric superpolynomials}

Consider $p_{a}=\sum_{\lfloor \beta,F\rfloor =\lfloor\lambda,E_{R}\rfloor }A(\beta,F) M_{\beta,F}$, where $\lfloor \lambda,E\rfloor $ is row-strict and $\boldsymbol{T}_{i}p_{a}=-p_{a}$ for $1\leq i<N$. Recall the action of $\boldsymbol{T}_{i}$ on the sum, which decomposes into pairs and singletons. Suppose $\lfloor \beta,F\rfloor =\lfloor \lambda,E_{R}\rfloor $ for some
$\beta$ with $\beta_{i}<\beta_{i+1}$, some $i$. Let $z=\zeta_{\beta,F}(i+1) /\zeta_{\beta,F}(i) $ then by~(\ref{TiM1})%
\[
(\boldsymbol{T}_{i}+1) \big(A(\beta,F)M_{\beta,F}+A(s_{i}\beta,F) M_{s_{i}\beta,F}\big) =0
\]
implies $A(\beta,F) =-\frac{1-zt}{1-z}A(s_{i}\beta,F)$. Recall $\sigma(n) =(-1)^{n}$
(Definition~\ref{defMD}).

\begin{Lemma}
If $\boldsymbol{T}_{i}p_{a}=-p_{a}$ for $1\leq i<N$ then
\[
A(\beta,F) =\sigma(\operatorname{inv}(\beta)) \mathcal{R}_{1}(\beta,F) A(\beta^{+},F).
\]
\end{Lemma}

\begin{proof}
Suppose $\beta_{i}>\beta_{i+1}$ then%
\begin{gather*}
\frac{\mathcal{R}_{1}(s_{i}\beta,F)}{\mathcal{R}_{1}(\beta,F)}
=u_{1}\bigg(\frac{\zeta_{\beta,F}(i)}{\zeta_{\beta,F}(i+1)}\bigg) .
\tag*{\qed}
\end{gather*}
\renewcommand{\qed}{}
\end{proof}

Consider the possibilities when $\beta_{i}=\beta_{i+1}$ and $j=r_{\beta}(i)$:
$(i)$~if $c(j,F) =c(j+1,F)-1$, that is, $j$ and $j+1$ are in adjacent cells of column 1 of $Y_{F}$ then
$\boldsymbol{T}_{i}M_{\beta,F}=-M_{\beta,F}$, imposing no conditions on
$A(\beta,F) $; $(ii)$~if $c(j,F) =c(j+1,F) +1$ then $\boldsymbol{T}_{i}M_{\beta,F}=tM_{\beta,F}$ but this occurs only if there are adjacent equal values ($\beta_{i}$) in row 1 of
$\lfloor \lambda,E_{R}\rfloor$, ruled out by hypothesis; $(iii)$~$c(j,F) <0<c(j+1,F)$. In~this case we relate $M_{\beta,F}$ to $M_{\beta,s_{j}F}$, where $\operatorname{inv}(s_{j}F)
=\operatorname{inv}(F) -1$: the formulas similar to (\ref{TiM2}) with
$z=\zeta_{\beta,F}(i)/\zeta_{\beta,F}(i+1)=t^{c(j,F) -c(j+1,F)}$ appear here: then%
\[
(\boldsymbol{T}_{i}+1)\big(A(\beta,F)M_{\beta,F}+A(\beta,s_{j}F) M_{\beta,s_{j}F}\big) =0
\]
implies $A\big(\beta,s_{j}F\big) =-\frac{1-tz}{1-z}A(\beta,F)$.

\begin{Lemma}
$\sigma(\operatorname{inv}(F)) \dfrac{A(\beta,F)}{\mathcal{C}_{1}(F)}
=\sigma\big(\operatorname{inv}(E_{S})\big)
\dfrac{A\big(\beta,E_{S}\big)}{\mathcal{C}_{1}\big( E_{S}\big)}$.
\end{Lemma}

Thus%
\begin{align*}
A(\beta,F) & =\sigma(\operatorname{inv}(\beta)) \mathcal{R}_{1}(\beta,F) A(\beta^{+},F)
\\
& =\sigma\big(\operatorname{inv}(\beta) +\operatorname{inv}(F) +\operatorname{inv}\big(E_{S}\big) \big) A\big(\lambda,E_{S}\big) \dfrac{\mathcal{R}_{1}(\beta,F) \mathcal{C}_{1}\big( E_{S}\big)}{\mathcal{C}_{1}(F)}.
\end{align*}

\begin{Theorem}
Suppose $\lambda\in\mathbb{N}_{0}^{N,+}$, $E\in\mathcal{Y}_{0}$, and
$\lfloor \lambda,E\rfloor $ is row-strict then
\[
p_{\lambda,E}^{a}=\sum_{\lfloor \alpha,F\rfloor \in\mathcal{T}(\lambda,E)} \sigma\big(\operatorname{inv}(\beta)
+\operatorname{inv}(F) +\operatorname{inv}\big(E_{S}\big) \big)
\frac{\mathcal{C}_{1}\big( E_{S}\big) \mathcal{R}_{1}(\alpha,F)}{\mathcal{C}_{1}(F) }M_{\alpha,F}%
\]
is the antisymmetric polynomial in $\mathcal{M}(\lambda,E)$,
unique when the coefficient of $M_{\lambda,E_{S}}$ is $1$.
\end{Theorem}

The antisymmetrizing operator is defined analogously to $S$.

\begin{Definition}
For $n\geq1$ let $X_{0}^{a}=1$ and $X_{n}^{a}=1-\frac{1}{t}\boldsymbol{T}_{n}X_{n-1}$, and $A^{(n)}=X_{1}^{a}X_{2}^{a}\cdots X_{n}^{a}$.
\end{Definition}

Equivalently $X_{n}^{a}=1-\frac{1}{t}\boldsymbol{T}_{n}+\frac{1}{t^{2}}\boldsymbol{T}_{n} \boldsymbol{T}_{n-1}+\dots+\frac{(-1)^{n}}{t^{n}}\boldsymbol{T}_{n}\cdots \boldsymbol{T}_{2}\boldsymbol{T}_{1}$.

\begin{Theorem}
If $1\leq j\leq n$ then $(\boldsymbol{T}_{j}+1) A^{(n)}=0$.
\end{Theorem}

\begin{proof}
The operators $\big\{{-}\frac{1}{t}\boldsymbol{T}_{i}\big\} $ satisfy the
braid relations so the same approach as in Theorem~\ref{symmop} works here,
and the proof then follows from $(\boldsymbol{T}_{i}+1) \big(1-\frac{1}{t}\boldsymbol{T}_{i}\big) =0$.
\end{proof}

Similarly to Corollary~\ref{S2S} one can show that%
\[
A^{(N-1)}A^{(N-1)}=t^{-N(N-1)/2}[N]_{t}!A^{(N-1)}.
\]
There is a result analogous to Proposition~\ref{pfact}.

\begin{Lemma}
Suppose for some $i$ that $\boldsymbol{T}_{i}p(x;\theta)=-p(x;\theta)$ then $T_{i}p\left(x^{(i)};\theta\right) =-p\left(x^{(i)};\theta\right)$, where
$x_{i+1}^{(i) }=tx_{i}^{(i)}$.
\end{Lemma}

\begin{proof}
By hypothesis%
\begin{align*}
(t-\boldsymbol{T}_{i}) p(x;\theta) =(t+1) p(x;\theta)
 =tp(x;\theta) -(1-t) x_{i+1}\frac{p(x;\theta) -p(xs_{i};\theta)}{x_{i}-x_{i+1}}T_{i}p(xs_{i};\theta).
\end{align*}
Substitute $x_{i+1}=tx_{i}$ in the equations:
\begin{align*}
(t+1) p(x;\theta) =tp(x;\theta)-t(p(x;\theta) -p(xs_{i};\theta))
-T_{i}p(xs_{i};\theta) =(t-T_{i}) p( xs_{i};\theta)
\end{align*}
and this shows $( 1+T_{i}) p(x;\theta) =0$ at
$x=x^{(i) }$.
\end{proof}

Suppose $p_{\lambda,E}^{a}$ is antisymmetric and $E_{0}=\{N-m,N-m+1,\dots,N\}$ and consider $p_{\lambda,E}^{a}(z)$, where $z=\big(z_{1},z_{2},\dots,z_{N-m-1},t^{-m},\dots,t^{-2},t^{-1},1\big)$, then by the lemma $T_{i}p(z;\theta)=-p(z;\theta)$ for $N-m\leq i<N$ which $\omega_{i}p(z;\theta) =t^{i-N}p(z;\theta)$ for $N-m\leq i\leq N$.
The eigenvalues determine~$\tau_{E_{0}}$ and thus $p(z;\theta)=\widetilde{p}(z) \tau_{E_{0}}(\theta)$. If range $1\leq i\leq N-m-2$ then $T_{i}\tau_{F}=t\tau_{F}$ thus $\widetilde{p}$
satisfies the equation%
\begin{gather*}
-\widetilde{p}(z) \tau_{E_{0}} =(1-t)z_{i+1}\frac{\widetilde{p}(z) -\widetilde{p}(zs_{i}) }{z_{i}-z_{i+1}}\tau_{E_{0}}+t\widetilde{p}(zs_{i}) \tau_{E_{0}},
\\
\widetilde{p}(zs_{i}) =\frac{z_{i}-tz_{i+1}}{z_{i+1}-tz_{i}}\widetilde{p}(z).
\end{gather*}
This implies $(z_{i+1}-tz_{i}) $ is a factor of $\widetilde{p}(z)$ and $\frac{\widetilde{p}(z)}{z_{i+1}-tz_{i}}$ is $s_{i}$-invariant. Furthermore $(tz_{i}-z_{j})$ is a~factor of $\widetilde{p}(z)$ for $1\leq i<j\leq N-m-1$. Also $z_{N-m-1}=t^{-m-1}$ implies $\widetilde{p}(z) =0$ (or else $\omega_{N-m-1}\tau_{E_{0}}=t^{-m-1}\tau_{E_{0}}$, contra) and so $\big(t^{m+1}z_{N-m-1}-1\big)$ is a factor of~$\widetilde{p}(z) $. Thus%
\[
p^{a}(z;\theta) =\prod\limits_{1\leq i<j\leq N-m-1}(
tz_{i}-z_{j}) \prod\limits_{k=1}^{N-m-1}\big( t^{m+1}z_{k}-1\big)
p_{0}( z_{1},\dots,z_{N-m-1}) \tau_{E_{0}}
\]
and $p_{0}$ is $\mathcal{S}_{N-m-1}$-symmetric. With methods similar to those
of Theorem~\ref{symnorm} and by use of the antisymmetrizing operator
$A^{(N-1) }$ one can derive a formula for $\Vert
p_{\lambda,E}^{a}\Vert^{2}$.

\section{Conclusion}

We constructed a representation of the Hecke algebra $\mathcal{H}_{N}(t)$ on superpolynomials and applied the theory of vector-valued
nonsymmetric Macdonald polynomials to this situation. The basic facts such as
orthogonal bases for irreducible representations on fermionic variables, the
partial order on compositions used in expressions for the Macdonald
polynomials, and a sketch of the Yang--Baxter graph technique for constructing
the polynomials starting from degree zero were presented. The polynomials are
mutually orthogonal with respect to a bilinear form in which the generators of
the Hecke algebra are self-adjoint. The ideas of Baker and Forrester were used
to construct symmetric polynomials and to determine their squared norms.

There are some topics which deserve further investigation. What can be proven
about values of the nonsymmetric Macdonald polynomials at special points such
as $\big( t^{N-1},t^{N-2},\dots,t,1\big) $ (see~Remark~\ref{conjval})?
Are there special values of the symmetric and anti-symmetric polynomials? A~minimal factorization was proven in Proposition~\ref{pfact}.

Characterizing singular values of the parameters $q$, $t$ and the corresponding
polynomials is another important problem: this means that for a specific value
of $\left( q,t\right) $ there is a polynomial annihilated by $\boldsymbol{D}%
_{i}$ for $1\leq i\leq N$ (see Definition~\ref{defD}). This problem is
connected with the existence of maps between different modules. Also there
should be interesting factorizations. Here are two examples with $N=5$.

Let $\alpha=(2,0,0,0,0)$, $m=2$, $E=\{3,4,5\}\in\mathcal{Y}_{0}$ then $\boldsymbol{D}_{i}M_{\alpha,E}=0$ for $1\leq i\leq5$
when $q^{2}t^{5}=1$ or $qt=-1$ (that is $q^{2}t^{2}=1$, $qt\neq1$), and
\begin{gather*}
M_{\alpha,E}\big(x_{1},x_{2},tx_{2},t^{2}x_{2},t^{3}x_{2}\big)
=t^{10}(tx_{1}-x_{2}) (qtx_{1}-x_{2}) \tau_{E},
\\
\tau_{E} =t^{4}\theta_{3}\theta_{4}-t^{3}\theta_{3}\theta_{5}+t^{2}\theta_{4}\theta_{5},
\end{gather*}
when $(q,t) $ takes on a singular value (note if $q=-1/t$ then
$qtx_{1}-x_{2}=-(x_{1}+x_{2})$).

Let $\alpha=(2,0,0,0,0)$, $m=3$, $E=\{1,2\}\in\mathcal{Y}_{1}$ then $\boldsymbol{D}_{i}M_{\alpha,E}=0$ for $1\leq i\leq5$ when $q^{2}t^{-5}=1$ or $q=-t$ (that is $q^{2}t^{-2}=1$ and $qt^{-1}\neq1$), and
\begin{gather*}
M_{\alpha,E}\big(x_{1},x_{2},t^{-1}x_{2},t^{-2}x_{2},t^{-3}x_{2}\big)
 =t^{6}\big( t^{-1}x_{1}-x_{2}\big) \big(qt^{-1}x_{1}-x_{2}\big)\tau_{E},
\\
\tau_{E} =\theta_{1}\theta_{2}(\theta_{3}+\theta_{4}+\theta_{5}) ,
\end{gather*}
when $q^{2}=t^{5}$ (set $q=u^{5}$, $t=u^{2}$) or $q=-t$.

Obviously there are delicate interactions among $\alpha$, $E$, $q$, $t$, $x$ for such factorizations to hold.

\pdfbookmark[1]{References}{ref}
\LastPageEnding

\end{document}